\documentclass[11pt,a4paper]{article}
\usepackage[dvips]{graphicx}
\usepackage{amssymb,latexsym,amsmath}
\usepackage{psfrag}
\usepackage[active]{srcltx} 
\usepackage{subcaption}
\usepackage{graphicx}
\usepackage{tabularx}
\usepackage{multirow}
\input epsf

\usepackage[ansinew]{inputenc} 

\usepackage{color}  

\newtheorem{prop}{Proposition}[section]
\newtheorem{theo}[prop]{Theorem}

\newtheorem{ex}[prop]{Example}

\newtheorem{rem}[prop]{Remark}

\newtheorem{lem}[prop]{Lemma}

\newenvironment{proof}
 {\begin{trivlist} \item[\hskip \labelsep {\bf Proof}\hspace*{3 mm}]}
 {\hfill$\Box$\end{trivlist}}

\newcommand{\R}{\mathbb{R}}

\newcommand{\sgn}{\operatorname{sgn}}

\def \bx {{\bf x}}
\def \bz {{\bf z}}
\def \bn {{\bf n}}
\def \bs {{\bf s}}
\def \bt {{\bf t}}

\def \Evx {\mathcal{E}(\bx)}

\newcommand{\te}{\theta}

\begin{document}

\title{Helicoidal surfaces of frontals in  Euclidean space as deformations of surfaces of revolution, with singularities}
\author{Luciana F. Martins  and
Samuel P. dos Santos \footnote {Work  supported by process no.  2022/10370-9, Fundação de Amparo à Pesquisa do Estado de São Paulo (FAPESP).} }

\maketitle
 \begin{abstract}
We investigate helicoidal surfaces in three-dimensional Euclidean space whose profile curves are frontals. Using the framework of Legendre curves and framed surfaces, we establish conditions under which helicoidal surfaces generated by frontals are themselves frontals or fronts. We then derive curvature expressions in terms of the invariants of the generating Legendre curve. Our study extends classical results on parallel and focal surfaces of surfaces of revolution to the helicoidal setting. In particular, we show that both parallel and focal surfaces of a helicoidal surface are helicoidal, with their generating curves arising from one-parameter deformations of the corresponding parallel and evolute curves. We prove that singularities of these curves persist under such deformations, revealing geometric rigidity and stability of singularities. Finally, we examine the behavior of the Gaussian and mean curvatures near singular points of helicoidal surfaces.
\end{abstract}

\renewcommand{\thefootnote}{\fnsymbol{footnote}}
 \footnote[0]{Keywords and Phrase.
Helicoidal surface, framed surface, Legendre curve,  surface of revolution, frontal }

\section{Introduction}

The study of helicoidal surfaces with singularities in 3-dimensional Euclidean space arises  naturally as a generalization of the classical theory of surfaces of revolution. These surfaces have recently attracted considerable attention (see, for example, \cite{Ha-Ho-Mo,msst,Na-Sa-Shi-Ta,TT}). Also known as generalized helicoids, they form a broader class that includes both helicoids and surfaces of revolution as special cases. 

Let \( l \) be a line lying in a plane \( \Pi \) in \( \mathbb{R}^3 \), and let \( \mathcal{C} \) be a curve contained in \( \Pi \). Suppose \( \mathcal{C} \) undergoes a combined motion in \( \mathbb{R}^3 \): it rotates around \( l \) while simultaneously translating parallel to \( l \), with the rate of translation proportional to the rate of rotation. The resulting set of points \( \mathcal{M} \) is called a \textit{generalized helicoid} (or \textit{helicoidal surface}) generated by \( \mathcal{C} \), which is referred to as the \textit{profile curve} of \( \mathcal{M} \). The line \( l \) is referred to as the \textit{axis} of \( \mathcal{M} \), and the ratio of the translation speed to the rotational speed is called the \textit{slant} of \( \mathcal{M} \). A surface of revolution corresponds to a helicoidal surface with slant $0$ (cf. \cite{Gray}), while the classical helicoid is a helicoidal surface whose profile curve is a straight line segment. 

In a surface of revolution, the geometry arises  purely from the circular motion of the profile curve around a fixed axis. In contrast, helicoidal surfaces result from a helical motion, involving both rotation and translation along the axis, producing a screw-like geometry. 

Helical motions, which generate helicoidal surfaces, are a particular class of screw motions as formalized in the foundational work of Ball (\cite{Ball}). Screw theory has become  a tool in fields such as robotics, multibody mechanics, computational geometry, and structural kinematics, where the analysis of rigid body motions requires the interplay of rotation and translation. From a geometric perspective, helicoidal surfaces provide a natural setting to investigate how such coupled motions influence curvature properties, simmetry, and the appearance of singularities.

In this paper, we investigate  helicoidal surfaces with singular points.  A frontal is a class of curves or surfaces that may admit singularities, and over the past decades, frontals have been extensively studied from the perspective of differential geometry. In our setting,  we consider frontals as singular plane curves (more precisely, Legendre curves, whose curvatures are given in \cite{FuTa-lc1}), and we treat the associated helicoidal surfaces as singular surfaces modeled by framed base surfaces,  whose basic invariants are introduced in  \cite{FuTa-Fs}. 

In Section 2, we review the theories of Legendre curves in the unit tangent bundle of the Euclidean plane $\R^2$, as well as the theory of framed surfaces in the Euclidean space $\R^3$. We also recall the definitions of parallel and evolute curves of Legendre curves, along with parallel and focal surfaces of  framed surfaces, with will play role in subsequent sections. 

Unlike surfaces of revolution, which are always frontals  when their profile curves are frontals (\cite{TT}), helicoidal surfaces require an additional condition to be considered frontals. In Section 3, we establish a condition under which helicoidal surfaces are themselves frontals, and we analyse when such surfaces are in fact fronts. Under this condition, the resulting helicoidal surface becomes a framed base surface. 
In Section 4, we give the curvatures and basic invariants of helicoidal surfaces by using the curvature of  the generating Legendre curves. 

Sections 5 and 6 are devoted to the study of the relationship between a helicoidal surface and associated parallel and focal surfaces.
It is well known that, for regular surfaces of revolution generated by a regular curve $\gamma$, the  parallel  surface is itself a surface of revolution, generated by a parallel curve of $\gamma$. This property also holds for surfaces of revolution generated by Legendre curves. For focal surfaces, a similar behavior occurs: one of the focal surfaces  coincides with  the surface of revolution generated by the evolute of  the profile curve, while the other  degenerates, lying entirely along the  axis of rotation (\cite{TT}).  

In this work, we extend this geometric viewpoint to helicoidal surfaces, where the additional twist introduced by the helicoidal motion leads to more intricate behaviors.
We show in Theorem \ref{theo:paralHelic} that  parallel surfaces of a helicoidal surface are also  helicoidal surfaces. However, their generating curve are no longer a simple radial offset of the original curve, as is the case of surfaces of revolution. Instead, it results in a 1-parameter deformation that depends on the twisting nature of the helicoidal motion,  encoded in the slant parameter of the helicoidal surface. This deformation process is further analyzed in Theorem~\ref{theo:def-par}, where we show that singularities of the original parallel curve persist under the helicoidal deformation, revealing a form of geometric rigidity of singular points.

Section~6 focuses on  the focal surfaces of a helicoidal surface. Proposition~\ref{prop:focal:gen} establishes that one of the focal surfaces of a helicoidal surface is itself helicoidal, with a space profile curve explicitly determined also by the invariants of the generating Legendre curve. Building on this and on the construction of plane profile curves from Section~5, we describe in Remark~\ref{rem:Lem} a  family of plane curves  associated with the focal surface. This family defines a 1-parameter deformation of the evolute of the profile curve. In Theorem~\ref{theo:def-foc}, we prove that, under suitable assumptions, the singularities of the evolute persist under this deformation. This result reveals a geometric stability of singularities in the focal geometry of helicoidal surfaces.

 Finally, in Section 7, we investigate the boundedness of  the Gaussian and  mean curvatures of helicoidal surfaces near singular points.

A natural question regarding helicoidal surfaces is the characterization of their singularities. This was announced in \cite{Na-Sa-Shi-Ta} by N. Kakatsuyama, K. Saji, R. Shimada, and M. Takahashi, where the authors study helicoidal surfaces by considering them as generalized framed base surfaces (as in \cite{TY}) and then use this framework to obtain results such as the basic invariants and curvatures, after showing that the helicoidal surface is also a framed base surface under a mild condition (Proposition 3.3). Criteria for the singularities of helicoidal surfaces is given in Theorem 4.3, and as a consequence, the authors show that certain types of singular points cannot occur. We believe that our paper and theirs were developed simultaneously and independently, and are complementary to each other.

Parts of this work were presented at the 7th International Workshop on Singularities in Generic Geometry and Applications, held in Fortaleza, Brazil, in 2023, and at the 18th International Workshop on Real and Complex Singularities, held in Valencia, Spain, in 2024. The first workshop was organized in honor of Farid Tari, and the second in honor of Juan José Nuño-Ballesteros, both celebrating their 60th birthdays. We dedicate this work to these two outstanding mathematicians, whose contributions have greatly enriched the field.

\section{Preliminaries}
\label{sec:preli}

In this paper, all maps and manifolds are differentiable of class \( C^\infty \) unless stated otherwise. Let \( \mathbb{R}^n \) denote the \( n \)-dimensional Euclidean space equipped with the inner product \( \mathbf{a} \cdot \mathbf{b} = a_1b_1 + \dots + a_nb_n \), where \( \mathbf{a} = (a_1, \dots, a_n) \) and \( \mathbf{b} = (b_1, \dots, b_n) \). We shall focus on the cases where \( n = 2 \) or \( n = 3 \). The norm of \( \mathbf{a} \) is given by \( |\mathbf{a}| = \sqrt{\mathbf{a} \cdot \mathbf{a}} \), and for \( n = 3 \), the vector product of \( \mathbf{a} \) and \( \mathbf{b} \) is denoted, as usual, by \( \mathbf{a} \times \mathbf{b} \).

We quickly review below main definitions and properties of Legendre curves in the Euclidean plane  $\R^2$ (cf. \cite{FuTa-lc1,FuTa-lc2}) and framed surfaces in the  Euclidean space $\R^3$ (cf. \cite{FuTa-Fs}). 

\subsection{Legendre curves} \label{sec:leg-curves}

Let \( I \subset \mathbb{R} \) be an open interval, and let \( S^1 \) be the unit circle in $\R^2$. Let \( \gamma: I \to \mathbb{R}^2 \) and \( \nu: I \to S^1 \) be  smooth mappings. We say that \( (\gamma, \nu): I \to \mathbb{R}^2 \times S^1 \) is a {\it Legendre curve} if \( \dot{\gamma}(t) \cdot \nu(t) = 0 \) for all \( t \in I \), where \( \dot{\gamma}(t) = (d\gamma/dt)(t) \). A point \( t_0 \in I \) such that \( \dot{\gamma}(t_0) = 0 \) is called a {\it singular point} of \( \gamma \). If a Legendre curve $ (\gamma, \nu)$ is an immersion, it is called a {\it Legendre immersion}. We say that \( \gamma: I \to \mathbb{R}^2 \) is a {\it frontal} (resp. {\it front}) if there exists a smooth map \( \nu: I \to S^1 \) such that \( (\gamma, \nu): I \to \mathbb{R}^2 \times S^1 \) is a Legendre curve (resp. Legendre immersion).

Given a Legendre curve $(\gamma, \nu): I \to \mathbb{R}^2 \times S^1$, let $J$ be the anticlockwise rotation by angle $\pi/2$ in $\mathbb{R}^2$, and let $\mu(t) = J(\nu(t))$. Then $\{\nu(t), \mu(t)\}$ is a moving frame along the frontal $\gamma(t)$ in $\mathbb{R}^2$, and we have a Frenet-type formula given by

\begin{equation}\label{eq:LegCurvatures}
	\left( \begin{array}{l}
		\dot{\nu}(t)\\ \dot{\mu}(t)
	\end{array}\right)= \left( \begin{array}{cc}
		0 & \ell(t)\\ -\ell(t) & 0 
	\end{array}\right)  \left( \begin{array}{l}
		\nu(t) \\ \mu(t)
	\end{array}\right), \quad \dot{\gamma}(t) = \beta(t) \mu(t),
\end{equation}
where $\ell(t) = \dot{\nu}(t) \cdot \mu(t)$ and $\beta(t) = \dot{\gamma}(t) \cdot \mu(t)$. The pair $(\ell, \beta)$ is called the {\it curvature of the Legendre curve} $(\gamma, \nu)$. It holds that a Legendre curve $(\gamma, \nu)$ is a Legendre immersion if and only if $(\ell,\beta) \neq (0,0)$. Notice that $\gamma$ is singular at $t$ if and only if $\beta(t) = 0$.

Let $(\gamma, \nu): I \rightarrow \mathbb{R}^2 \times S^1$ be a Legendre curve. Given $\lambda \in \R$, the {\it parallel curve} $\gamma_\lambda: I \rightarrow \mathbb{R}^2$ of $\gamma$ is defined by  
$\gamma_\lambda (t) = \gamma(t) + \lambda \nu(t)$, and we note that $(\gamma_\lambda, \nu)$ is also a Legendre curve. In a neighborhood of a regular point of $\gamma$, the evolute of $\gamma$  is given by $\mathcal{E}_{\gamma}(t) = \gamma(t) + \frac{1}{\kappa(t)} {\bf n}(t)$, away from points where $\kappa(t) = 0$, where $\kappa(t)$ is the usual curvature of a regular plane curve and ${\bf n}(t) = J(\dot{\gamma}(t)/|\dot{\gamma}(t)|) = - \sgn(\beta(t)) \nu(t)$.
Since the curvature may diverge at singular points of $\gamma$, to propertly address such points we regard $(\gamma, \nu)$ as a Legendre immersion and define its  {\it evolute}   $\mathcal{E}_{\gamma}: I \to \mathbb{R}^2$  by  $\mathcal{E}_{\gamma}(t) = \gamma(t) -  \frac{\beta(t)}{\ell(t)} \nu(t)$, according to Theorem 3.3 of \cite{FuTa-lc2}. This definition generalizes the classical evolute of regular plane curves, since at regular points of $\gamma$ we have $\ell(t)=|\beta(t)| \kappa(t)$ (Lemma 3.1 of \cite{FuTa-lc2}), and we assume throughout that $\kappa(t) \neq 0$.  

Given $\lambda \in \R$, it holds that the parallel curve $\gamma_{\lambda}$ of $\gamma$ is singular at $t_0 \in I$ if and only if either $\gamma$ is singular at $t_0$, or $\gamma$ is regular at $t_0$ and one of the following holds: $\lambda = 1/\kappa(t_0)$ if $\nu (t_0)= {\bf n}(t_0)$, or $\lambda = - 1/\kappa(t_0)$ if  $\nu(t_0) = - {\bf n}(t_0)$ (for details, see Proposition 2.8 of \cite{FuTa-lc2}, where both cases $\nu = \pm {\bf n}$ must the considered).  Consequently,   if $\gamma$ is regular at $t_0$, then  $\gamma_{\lambda}$ is singular at $t_0$ if and only if $\lambda = - \beta(t_0)/ \ell (t_0)$.

\subsection{Framed surfaces} \label{sec:framed-surf}

Considering Legendre curves, one has an orthonormal frame for $\mathbb{R}^2$ along the curve. The same approach holds for surfaces in $\mathbb{R}^3$. Let $U$ be a simply connected domain of $\mathbb{R}^2$, and let $S^2$ be the unit sphere in $\mathbb{R}^3$. We denote the 3-dimensional smooth manifold $\{(\mathbf{a}, \mathbf{b}) \in S^2 \times S^2; \, \mathbf{a} \cdot \mathbf{b} = 0\}$ by $\Delta$.

We say that $(\mathbf{x}, \mathbf{n}, \mathbf{s}): U \rightarrow \mathbb{R}^3 \times \Delta$ is a {\it framed surface} if $ \mathbf{x}_u(u,v) \cdot \mathbf{n}(u,v) = 0$ and $\mathbf{x}_v(u,v) \cdot \mathbf{n}(u,v) = 0$ for all $(u,v) \in U$, where $\mathbf{x}_u = \partial \mathbf{x} / \partial u$ and $\mathbf{x}_v = \partial \mathbf{x} / \partial v$. We say that $\mathbf{x}: U \to \mathbb{R}^3$ is a {\it framed base surface} if there exists $(\mathbf{n}, \mathbf{s}): U \rightarrow \Delta$ such that $(\mathbf{x}, \mathbf{n}, \mathbf{s})$ is a framed surface. A  framed surface $(\mathbf{x}, \mathbf{n}, \mathbf{s})$ is a {\it framed immersion} if $(\mathbf{x}, \mathbf{n}, \mathbf{s})$ is an immersion.  
A point $p \in U$ is a {\it singular point} of $\mathbf{x}$ if $\mathbf{x}$ is not an immersion at $p$.

The pair $(\bx, \bn) : U \to \R^3 \times S^2$ is said to be a {\it Legendre surface} if $\bx_u(u,v) \cdot \bn(u,v) = 0$ and $\bx_v(u,v) \cdot \bn(u,v) = 0$ for all $(u, v) \in U$. Moreover, when a  Legendre  surface $(\bx, \bn)$ is an immersion, this is called a {\it Legendre immersion}. We say that $\bx : U \to \R^3$ is a {\it frontal} (respectively, a {\it front}) if there exists a map $\bn : U \to S^2$ such that the pair $(\bx, \bn) : U \to \R^3 \times S^2$ is a Legendre surface (respectively, a Legendre immersion). Then, a  framed base surface is a frontal. At least locally, a frontal is a framed base surface.

Let $\bx : U \to \R^3$ be a frontal. The function $\Lambda:U\to\mathbb{R}$ given by $\Lambda(u,v)=\det(\bx_u,\bx_v,\bn)(u,v)$ is called the \textit{signed area density}. A singular point $p $ of a frontal is called  \textit{non-degenerate} if \(d\Lambda (p) \neq 0\).  In this case, the singular set of the frontal is a regular curve near $p$, called {\it singular curve}. A singular point $p \in U$ of a map $f: U \rightarrow \R^3$ is called a {\it cuspidal edge} if the map-germ $f$ at $p$ is $\mathcal{A}$-equivalent to $(u,v) \mapsto (u, v^2, v^3)$ at $0$. (Two map-germs $f_1,f_2: (\R^n,0) \rightarrow (\R^m,0)$ are $\mathcal{A}$-equivalent if there exist diffeomorphisms $S:(\R^n,0) \rightarrow (\R^n,0) $ and $T: (\R^m,0)  \rightarrow (\R^m,0)$ such that $f_2 \circ S = T \circ f_1$.) A singular point $p$ of $f$ is a {\it swallowtail} if $f$ at $p$ is $\mathcal{A}$-equivalent to $(u, u^2v + 3u^4, 2uv + 4u^3)$ at $0$.     Cuspidal edges and swallowtails are examples of non-degenerate singular points of frontals. A non-degenerate singular point $p$ of a frontal $f$ is called of the \textit{first kind} if the kernel of the differential of  $f$ at $p$ is not parallel to the singular curve.  Cuspidal edges are  examples of non-degenerated singular points of frontals of the first kind.

Let  $(\bx, \bn , \bs): U\rightarrow\R^3 \times \Delta$ be a framed surface  and $\bt(u,v) = \bn(u,v)\times\bs(u,v)$. Then $\{\bn(u,v),\bs(u,v),\bt(u,v)\}$ is a positive orthogonal moving frame along $\bx(u,v)$. Therefore, we have the following systems of differential equations:  
\begin{equation}\label{eq:g-framed-surface}
	\left(\begin{array}{cc}
		\bx_u  \\
		\bx_v                    
	\end{array}\right) = 
	\left(\begin{array}{cc}
		a_1  &  b_1\\
		a_2  &  b_2
	\end{array}\right)
	\left(\begin{array}{cc}
		\bs     \\
	\bt   
	\end{array}\right),
\end{equation}

\begin{equation}\label{eq:f-framed-surface}
	\left(\begin{array}{c}
		\bn_u  \\
		\bs_u   \\
		\bt_u
	\end{array}\right) = 
	\left(\begin{array}{ccc}
		0     &  e_1  & f_1 \\
		-e_1  &  0    & g_1 \\
		-f_1  &  -g_1 & 0
	\end{array}\right)
	\left(\begin{array}{c}
		\bn   \\
		\bs    \\
		\bt
	\end{array}\right), 
	\left(\begin{array}{c}
		\bn_v  \\
		\bs_v   \\
		\bt_v
	\end{array}\right) = 
	\left(\begin{array}{ccc}
		0     &  e_2  & f_2 \\
		-e_2  &  0    & g_2 \\
		-f_2  &  -g_2 & 0
	\end{array}\right)
	\left(\begin{array}{c}
		\bn   \\
		\bs    \\
		\bt
	\end{array}\right),
\end{equation}
where $a_i,b_i e_i,f_i,g_i: U\rightarrow\R$, $i=1,2$, are smooth functions. These functions are called {\it basic invariants} of the framed surface $(\bx,\bn,\bs)$. The square matrices given in  \eqref{eq:g-framed-surface}  and \eqref{eq:f-framed-surface} are denoted by $\mathcal{G}, \mathcal{F}_1$ and $\mathcal{F}_2$, respectively, and the triple $(\mathcal{G},\mathcal{F}_1,\mathcal{F}_2)$ is also called {\it basic invariants} of the framed surface $(\bx,\bn,\bs)$. These invariants satisfy the following conditions:

\begin{equation}\label{eq:int-fs}
	\left\{\begin{array}{l}
		{a_1}_v-b_1 g_2={a_2}_u-b_2 g_1 \\ {b_1}_v-a_2 g_1={b_2}_u-a_1 g_2 \\ a_1 e_2+b_1 f_2=a_2 e_1+b_2 f_1\end{array}\right. ,\ \ \ 
		\left\{\begin{array}{l}
		{e_1}_v-f_1 g_2={e_2}_u-f_2 g_1 \\ {f_1}_v-e_2 g_1={f_2}_u-e_1 g_2 \\ {g_1}_v-e_1 f_2={g_2}_u-e_2 f_1\end{array}\right.
\end{equation}
and  ${\mathcal{F}_2}_u - {\mathcal{F}_1}_v = \mathcal{F}_1\mathcal{F}_2 - \mathcal{F}_2\mathcal{F}_1$. 

Let $(\mathcal{G},\mathcal{F}_1,\mathcal{F}_2)$ be the basic invariants of $(\bx, \bn , \bs)$. The smooth mapping  $C_F = (J_F,K_F,H_F):U\rightarrow\R^3$ given by
\begin{equation*}
	\begin{aligned}
		J_F & =\operatorname{det}\left(\begin{array}{ll}
			a_1 & b_1 \\
			a_2 & b_2
		\end{array}\right), \quad K_F=\operatorname{det}\left(\begin{array}{ll}
			e_1 & f_1 \\
			e_2 & f_2
		\end{array}\right), \\
		H_F & =-\frac{1}{2}\left[\operatorname{det}\left(\begin{array}{ll}
			a_1 & f_1 \\
			a_2 & f_2
		\end{array}\right)-\operatorname{det}\left(\begin{array}{ll}
			b_1 & e_1 \\
			b_2 & e_2
		\end{array}\right)\right] 
	\end{aligned}
\end{equation*}
is called the {\it curvature of the framed surface}.
We notice that if $\bx$ is a regular surface, then  their Gaussian curvature $K$ and mean curvature $H$ satisfy $K=K_F/J_F$ and $H=H_F/J_F$.

The smooth mapping $I_F:U\rightarrow\R^8$ given by
	\begin{equation*}
	\begin{gathered}
		I_F=\left(C_F, \, \operatorname{det}\left(\begin{array}{ll}
			a_1 & g_1 \\
			a_2 & g_2
		\end{array}\right), \, \operatorname{det}\left(\begin{array}{ll}
			b_1 & g_1 \\
			b_2 & g_2
		\end{array}\right),\right. \\
		\left.\operatorname{det}\left(\begin{array}{ll}
			e_1 & g_1 \\
			e_2 & g_2
		\end{array}\right), \, \operatorname{det}\left(\begin{array}{ll}
			f_1 & g_1 \\
			f_2 & g_2
		\end{array}\right), \, \operatorname{det}\left(\begin{array}{ll}
			a_1 & e_1 \\
			a_2 & e_2
		\end{array}\right)\right) 
	\end{gathered}
	\end{equation*}
is called the {\it concomitant mapping} of the framed surface $(\bx, \bn, \bs)$.

\begin{prop}\label{prop:FuTa-Fs} {\rm(\cite{FuTa-Fs}) } Let $(\bx, \bn , \bs): U\rightarrow\R^3 \times \Delta$ be a framed surface and $p \in U$.
	\begin{enumerate}
		\item[\rm (a)] $\bx$  is an immersion (a regular surface) around $p$ if and only if $J_F ( p) \neq 0$.
		\item[\rm (b)]  $(\bx,\bn)$ is a Legendre immersion around $p$ if and only if $C_F(p) \neq  0$.
	\item[\rm (c)] $(\bx, \bn , \bs)$ is a framed immersion around $p$ if and only if $I_F(p) \neq 0$.
	\end{enumerate}
\end{prop}

Next, we consider parallel  and focal surfaces of framed surfaces. For a framed surface $(\bx, \bn , \bs): U\rightarrow\R^3 \times \Delta$, a {\it parallel surface} $\bx^\lambda : U \to \R^3$ is given by $\bx^\lambda (u, v) = \bx(u, v) + \lambda \bn(u, v)$, where $\lambda \in \R$. It holds that $(\bx^\lambda, \bn , \bs)$ is a framed surface. A {\it focal surface} of the framed surface is given by $\Evx : U \to \R^3$, with  $\Evx (u, v) = \bx(u, v) + \lambda \bn(u, v)$, where $\lambda$ is a solution of the equation 
\begin{equation}\label{eq:solutionFS}
 K_F(u,v)\lambda^2 - 2H_F(u,v)\lambda + J_F(u,v) = 0. 
\end{equation}  
We note that these definitions are generatizations of those ones for regular surfaces. For instance, in a neigborhood of a regular point of $\bx$,  it holds that $\lambda$ is a solution of \eqref{eq:solutionFS} if and only if $1/\lambda$ is a solution of $\lambda^2 - 2H\lambda + K = 0$. Focal surfaces of a Legendre surface are also Legendre surfaces. The study of parallel and focal surfaces of fronts  from a different viewpoint is given in \cite{Tera-parallel, Tera-principal, Tera-focal}.
For more details on framed surfaces, including existence and uniqueness theorems, see \cite{FuTa-Fs}.

\section{Helicoidal surfaces of frontals}

Let $I\subset \R$ be an open interval  and $(\gamma, \nu):  I\to\R^2$ be a Legendre curve with curvature  $(\ell, \beta)$. 
In the following we identify $\R^2$ with the $xz$-plane  into $xyz$-space $\R^3$ and we deal with the plane curve $\gamma (t) = (x(t), z(t))$ as a curve in the $xz$-plane. Considering the moving frame $\{\nu(t), \mu(t) \}$ along $\gamma(t)$, let $\nu(t)= (a(t),b(t))$. Then  $\dot{x}(t)a(t) + \dot{z}(t)b(t) = 0$, $a^2(t) + b^2(t) = 1$ and  $\mu(t) = (-b(t),a(t))$. Let $\varphi: I \rightarrow \R$ be a smooth function such that $a(t)= \cos \varphi (t)$ and $b(t) = \sin \varphi (t)$. Then, by \eqref{eq:LegCurvatures}, it holds that
$\dot\gamma (t) = (\dot{x}(t),\dot{z}(t)) = \beta(t) (-b(t),a(t)) =(-\beta(t)\sin \varphi(t), \beta(t)\cos \,\varphi(t))$, $
\dot{\nu}(t)=(\dot{a}(t), \dot{b}(t)) =  \ell(t) (-b(t), a(t))$ and 
 $\ell(t) = \dot{\varphi}(t)$.

With the above notations, we shall consider helicoidal surfaces with profile curve $\gamma$ and axis being  the $x$-axis or the $z$-axis.
They are surfaces in $\R^3$ which  can be parametrized by $\bx,	{\bf z}: I \times \R \rightarrow \R^3$ given by
	\begin{equation}\label{eq:heligenx}
	{\bf x} (t,\theta)=(c\, \theta +x(t),z(t)\cos \theta ,z(t)\, \sin \theta)
	\end{equation}
and 
	\begin{equation}\label{eq:heligenz}
	{\bf z} (t,\theta)=(x(t)\cos \theta,x(t)\, \sin \theta, c\,\theta + z(t)),
\end{equation}
where $\bx$ (resp. $\bz$) is called the {\it helicoidal surface of $\gamma$ around $x$-axis (resp. $z$-axis)}.

Let us  first consider the helicoidal surface around the  $z$-axis given by \eqref{eq:heligenz}. The first order derivatives of $\bz$ are
$$\begin{array}{ll}
	\bz_t(t,\theta) = \beta(t)(-\sin\varphi(t)\cos\theta,-\sin\varphi(t)\sin\theta,\cos\varphi (t)) \ \ \text{and}\\
	\bz_\theta(t,\theta) = (-x(t)\sin\te,x(t)\cos\te,c).
\end{array}$$
 Then the coefficients of the first fundamental  form of $\bz$ are

\begin{equation*}\label{eq:coeficientesI}
	E(t,\theta)= \beta^2(t), \ \ \  F(t, \theta)=c\,\beta(t)\cos\varphi(t) \ \ \ {\rm and} \ \ \  G(t,\theta)=x^2(t)+c^2
\end{equation*}
and  $\bz$ is singular at $(t, \theta)$ if and only if 

\begin{equation*}\label{eq:regularpart}
	(EG-F^2)(t,\theta)=\beta^2(t)(x^2(t)+c^2)-c^2\beta^2(t)\cos^2\varphi(t)=\beta^2(t)\xi_\bz^2(t)
\end{equation*}
vanishes, where 
\begin{equation}\label{eq:xi:z}
	\xi_\bz(t)=\sqrt{c^2\sin^2\varphi(t)+x^2(t)}\,,
	\end{equation}
 i.e., at singular points of $\gamma$  ($\beta(t)=0$) or when $\xi_\bz(t)=0$. We note that 
$\xi_\bz(t)=0$ if and only if the profile curve intersects  the $z$-axis ($x(t)=0$) and, futhermore,  $\bz$ is a revolution surface ($c=0$) or $\nu(t)$ is orthogonal to the $z$-axis  ($b(t) = \sin \varphi(t)=0$) (see Figure \ref{fig:quadocsiseanula}).

\begin{figure}[h]
	\centering
\vspace*{-.4cm}	\includegraphics[width=10cm]{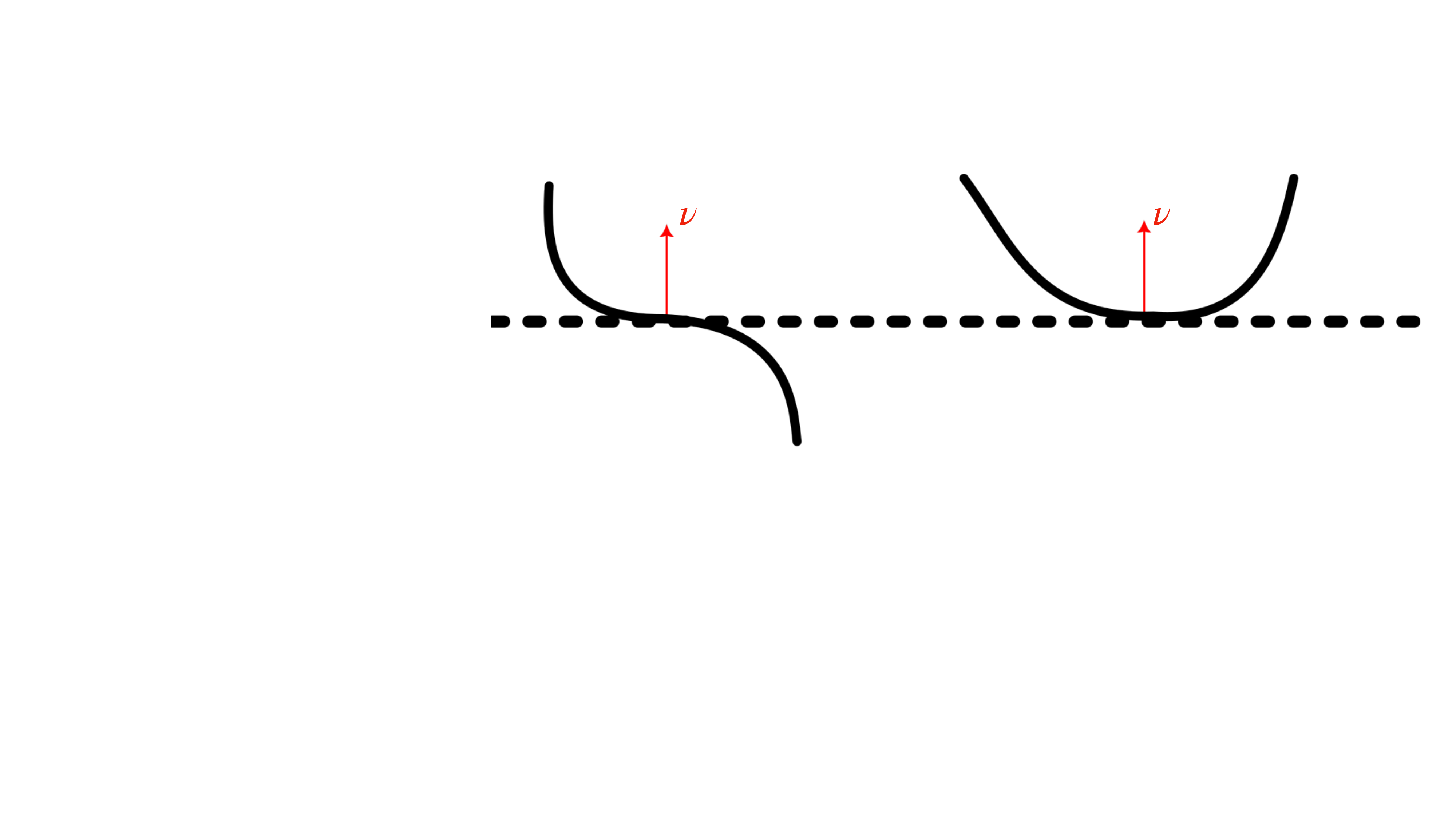}	\hspace*{2cm}	\vspace*{-2cm}
	\caption{Profile curves  of $\bz$ with  $\xi_\bz(t) = 0$, for some $t$. The dashed line is the axis of $\bz$.}
	\label{fig:quadocsiseanula}
\end{figure}


Considering the helicoidal  surface around the  $x$-axis with parametrization given by \eqref{eq:heligenx}, the same conclusion regarding the singular points of ${\bf x}$ holds,  where  
\begin{equation}\label{eq:xi:x}\xi_\bx(t)=\sqrt{c^2\cos^2 \varphi(t)+z^2(t)}\,.
		\end{equation}
	
	We observe that when $c\neq 0$,  either $\xi_\bx(t) \neq 0$ or $\xi_\bz (t) \neq 0$ holds.

A criterion for determining when a surface of revolution is a frontal or a front, based on its profile curve, was established in \cite{TT} (Proposition 3.3). In the result below, we follow a similar strategy, extending it to include surfaces of revolution for completeness and providing additional details.

\begin{prop}\label{prop:curvafrentesupfrente}
	Under the above notation,  the following holds:
	\begin{enumerate}
		\item[\rm (a)]  If   $c=0$, then both revolution surfaces ${\bf z}$  and ${\bf x}$ are frontals. Moreover,   $\bf{z}$  (resp.   $\bf{x}$) is a front if and only if $\gamma$ is a front and  $(x(t),\cos \varphi(t))\neq (0,0)$ (resp.$(z(t), \sin \varphi(t)) \neq (0,0)$), for all $t$. Consequently, either the surface of revolution  $\bz$  or $\bx$  is a front if and only if $\gamma$ is a front.
		\item[\rm (b)] If $c\neq0$  and $\xi_\bz(t) \neq0$ (resp. $\xi_\bx(t) \neq0$), for all $t$,  then the helicoidal surface $\textbf{\rm{\textbf{z}}}$  (resp. $\textbf{\rm{\textbf{x}}}$) is a frontal. Consequently, either the helicoidal surface  $\bf x$ or $\bf z$ is a frontal. 
	Moreover,  $\textbf{\rm{\textbf{z}}}$ (resp. $\textbf{\rm{\textbf{x}}}$)  is a front if and only if  $\gamma$ is a front and  $(x(t),\beta(t))\neq (0,0)$ (resp. $(z(t),\beta(t))\neq (0,0)$), for all $t$.
	\end{enumerate}
\end{prop}

\begin{proof}

We consider the helicoidal surface $\mathbf{z}$; the case of $\mathbf{x}$ is analogous.

By direct computation, we obtain:
%
\begin{multline*}
	(\textbf{z}_t \times \textbf{z}_\theta )(t, \theta) = \beta(t)(-c\,\sin \varphi (t)\sin\theta-x(t)\cos \varphi (t)\cos \theta,  c\,\sin \varphi (t) \cos \theta-x(t)\cos \varphi (t)\sin \theta,\\ -x(t)\sin\varphi (t))
\end{multline*}
and
\[
|(\textbf{z}_t \times \textbf{z}_\theta) (t, \theta)|^2 = (EG - F^2)(t, \theta) = \beta^2(t) \xi_\textbf{z}^2(t),
\]
where \(\xi_\textbf{z}\) is defined by \eqref{eq:xi:z}.

(a) If \(c = 0\), then a unit normal vector field to the surface \(\textbf{z}\) is given by
\[
\bn^\textbf{z}(t, \theta) = (\cos\varphi(t)\cos\theta, \cos\varphi(t)\sin\theta, \sin\varphi(t)).
\]

Consequently, \(\textbf{z}\) is a frontal.

Let \(L = (\textbf{z}, \bn^{\textbf{z}})\) with \(\bn^{\textbf{z}}\) as given above. Since \(\dot \varphi(t) = \ell(t)\), the differential  of $L$ at \((t, \theta)\) is given by the matrix:

\[
dL(t, \theta) = \begin{bmatrix}
	-\beta(t) \sin\varphi(t) \cos\theta & -x(t) \sin\theta \\
	-\beta(t) \sin \varphi(t) \sin\theta & x(t) \cos\theta \\
	\beta(t) \cos \varphi(t) & 0 \\
	-\ell(t) \sin\varphi(t) \cos\theta & -\cos\varphi(t) \sin\theta \\
	-\ell(t) \sin\varphi(t) \sin\theta & \cos\varphi(t) \cos\theta \\
	\ell(t) \cos\varphi(t) & 0
\end{bmatrix}.
\]

If \(\gamma\) is not a front at \(t_0\), then \((\ell(t_0), \beta(t_0)) = (0, 0)\). This leads to the first column of \(dL(t_0, \theta)\) vanishing, indicating that \(\textbf{z}\) is not a front at \((t_0,\theta)\). The same conclusion holds if \((x(t_0), \cos \varphi(t_0)) = (0, 0)\), for some \(t_0 \in I\).

Now, suppose that \(\gamma\) is a front, i.e., \((\ell(t), \beta(t)) \neq (0, 0)\), and \((x(t), \cos \varphi(t)) \neq (0, 0)\), for all \(t\). Under these conditions, the above matrix has rank 2 and, therefore, \(\textbf{z}\) is a front.

(b)  Suppose that $c\neq0$ and  $\xi_\bz(t) \neq 0$, for all $t$. So, it holds that 
	\begin{multline}
		\label{eq:nuz}
		\bn^\bz(t, \theta) = \dfrac{1}{\xi_\bz(t)}(c\,\sin \varphi(t) \sin\theta +x(t)\cos \varphi (t) \cos \theta, \\ -c\, \sin \varphi (t) \cos \theta +x(t)\cos \varphi  (t)\sin \theta, x (t)\sin\varphi (t))
\end{multline}
	is a unit vector field normal to $\textbf{z}$ and $\textbf{z}$ is a frontal. 
	
	Since either $\sin \varphi (t)\neq0$ or $x(t) \neq 0$ and, for $\bn^\bz$ given by \eqref{eq:nuz}, $dL(t,\theta)=d(\textbf{z},\bn^\bz)(t, \theta)$  is given by the matrix
	
	$$\left[\begin{array}{cc}
		-\beta(t) \sin\varphi (t) \cos\theta& -x(t) \sin\theta\\
		-\beta (t) \sin\varphi (t) \sin\theta& x(t) \cos\theta\\
		\beta (t) \cos \varphi  (t)& c\\
	A(t,\theta)  & \frac{c\sin\varphi (t)\cos\theta- x(t) \cos\varphi (t) \sin\theta}{\xi_\bz (t)}\\
		B(t,\theta) & \frac{c\sin\varphi (t) \sin\theta + x (t) \cos\varphi  (t) \cos\theta}{\xi_\bz (t)}\\
		\frac{-c^2 \beta (t) \sin^4\varphi  (t)+ x(t)^3 \ell (t) \cos\varphi (t)}{\xi_\bz^3 (t)}& 0
	\end{array}\right],$$
where 
\begin{multline*}
		A(t,\theta) = \frac{1}{\xi_\bz^3 (t)} \left( c\,\beta (t) \sin^2\varphi (t)(-c\sin\varphi (t) \cos\varphi  (t) \cos\theta +x  (t) \sin\theta) \right. \\ +
\left.	x(t) \,\ell(t)(c\, x (t) \cos\varphi (t) \sin\theta -(c^2 + z(t)^2) \sin\varphi (t) \cos\theta)\right)
\end{multline*}
and 
\begin{multline*} B(t,\theta)= \frac{1}{\xi_\bz^3 (t)} \left(-c\,\beta  (t)\sin^2\varphi (t)(c\sin\varphi (t) \cos\varphi (t) \sin\theta + x(t) \cos\theta ) \right.  \\ -  \left.
	x (t) \,\ell (t)(c\,x (t) \cos\varphi  (t) \cos\theta+ (c^2 + z(t)^2) \sin\varphi (t)\sin\theta) \right)\,.
\end{multline*}
The result follows simply by analyzing this matrix.
\end{proof}


\begin{rem}\label{rem:helifront} By Proposition \ref{prop:curvafrentesupfrente}, the helicoidal surface \(\bz\)  (resp.  \(\bx\)), with \(c \neq 0\) and  \(\xi_\bz(t) \neq 0\) (resp. \(\xi_\bx(t) \neq 0\)) is a front if and only if the profile curve is a front and is  nonsingular at the points where  it intersects the  axis of the surface (see Figure \ref{fig:quandohelnaofrente}).
		Moreover, we observe that when \(c = 0\), both revolution surfaces  \(\bx\) and 		\(\bz\) are frontals, 
 but fail to be fronts at \((t, \theta)\) if and only if  the profile curve is a frontal but not a front at \(t\). In the case of helicoidal surfaces, if the profile curve is a frontal but not a front, then both  \(\bx\) and \(\textbf{z}\) will also fail to be fronts. However, there are cases where the profile curve is a front, but both helicoidal surfaces are not fronts  (see Example \ref{ex:contraexemplofrentenaogerafrente}).
\end{rem}

\begin{figure}[h]
	\centering
	\vspace*{-.9cm}	\hspace*{1cm}	\includegraphics[width=15cm]{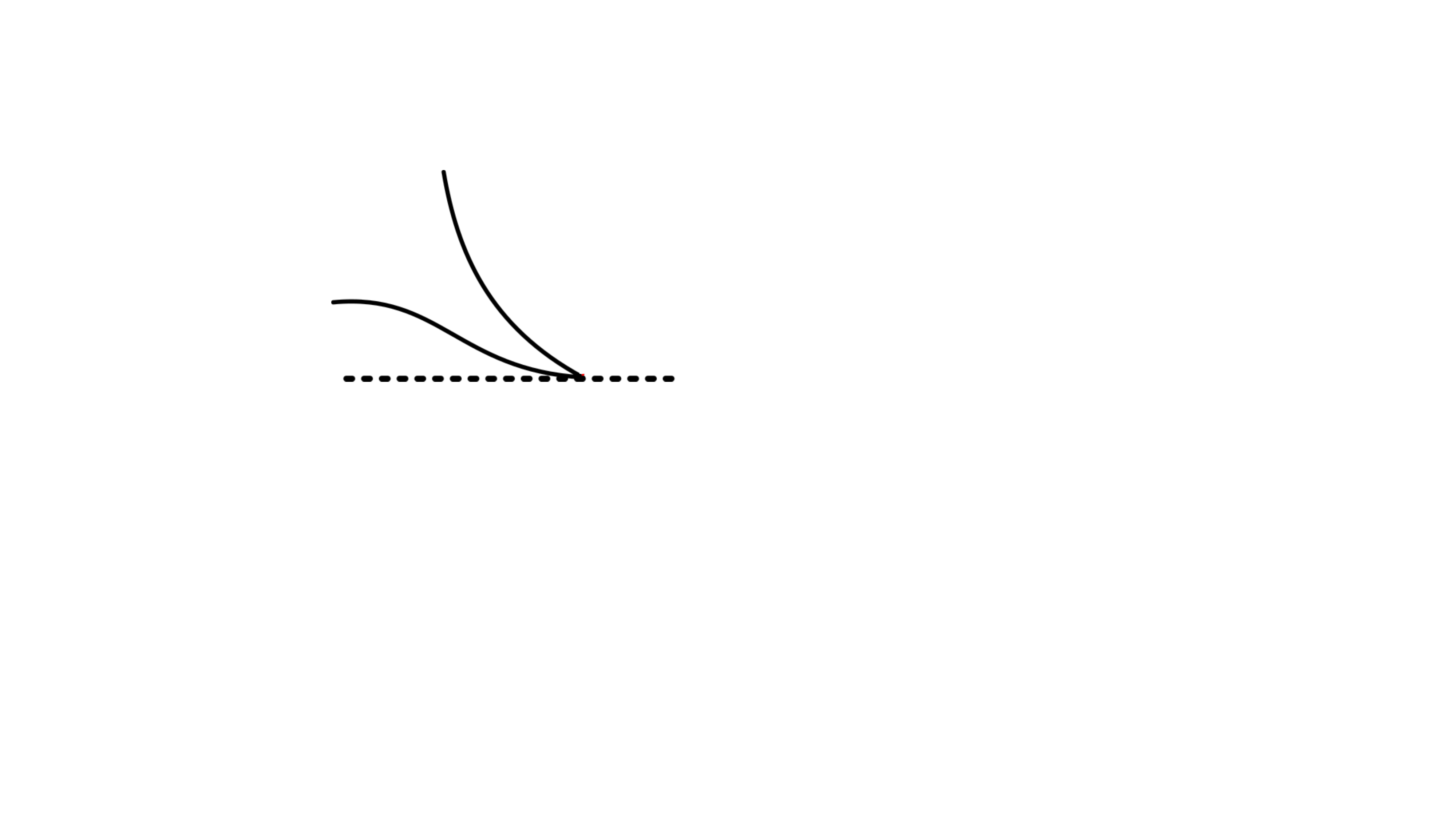}\vspace*{-4cm}
	\caption{An example illustrating a case where the profile curve could be a front, but the corresponding helicoidal surface  is not. The dashed line represents the axis of the helicoidal surface.}
	\label{fig:quandohelnaofrente}
\end{figure}

\begin{ex}\label{ex:contraexemplofrentenaogerafrente}
Consider the frontal  $$\gamma(t)=\left(\frac{(t+1) \cos t + (t-1) \sin t -1}{\sqrt{2}}
,\frac{(1-t) \cos t + (1 + t)\sin t - 1} {\sqrt{2}}
\right)$$ near $t=0$. Then $\varphi(t)=\arcsin\left(\frac{\cos t-\sin t}{\sqrt{2}}\right)$ and $\beta(t)=t$.  Since $\ell(0) = \dot{\varphi}(0) = 1$, it holds that $\gamma$ is a front around $t=0$. On the other hand, since $x(0)=z(0)=\beta(0)=0$, then by Proposition \ref{prop:curvafrentesupfrente},  the helicoidal surfaces $\bx$ and $\bz$ are both not fronts. \end{ex}

\section{Helicoidal surfaces as a framed surface}

In the remainder of this work, we will use the notation introduced in the previous sections.

Under a suitable condition, helicoidal surfaces parametrized by equations \eqref{eq:heligenx} or \eqref{eq:heligenz} are framed base surfaces. In what follows, we restrict our analysis to the surface defined by \eqref{eq:heligenz}, since the case of \eqref{eq:heligenx} is analogous.

\begin{prop} 	Let $(\gamma, \nu): I \rightarrow \R^2 \times S^1$  be a Legendre curve with curvature $(\ell,\beta)$, and let $\bz$ denote the  helicoidal surface generated by  $\gamma(t) = (x(t),z(t))$  around the $z$-axis  given by \eqref{eq:heligenz}. Suppose that $\xi_\bz(t)\neq0$, for all $t$, where  $\xi_\bz$ is defined by \eqref{eq:xi:z}.  Then 
	$(\bz, \bn^\bz, \bs^\bz): I \times \R \rightarrow  \R^3 \times \triangle$  is a framed surface whose basic invariants  $(\mathcal{G}^{\bz}, \mathcal{F}_1^{\textbf{z}}, \mathcal{F}_2^{\bz})$ at $(t, \theta)$  are given by: 	$$\mathcal{G}^{\bz}=\left[\begin{array}{cc}
		0 & -\beta(t)\\
		-\xi_\bz(t)\ & -c \cos\varphi(t)\
	\end{array}\right],$$
	$$\mathcal{F}_1^{\bz}=\dfrac{1}{\xi_\bz(t)}\left[\begin{array}{ccc}
		0 & \frac{c(\beta(t) \sin^2\varphi(t) + \ell (t) x(t) \cos\varphi(t))}{\xi_\bz(t)}& -\ell (t)x(t)\\
	\frac{-c\, (\beta(t) \sin^2\varphi(t) + \ell(t) x(t) \cos\varphi(t))}{\xi_\bz(t)}& 0 & c\,\ell(t) \sin\varphi(t) \\
		\ell(t)	x(t)& -c\, \ell (t)\sin\varphi(t)& 0
	\end{array}\right],$$
	and
	$$\mathcal{F}_2^{\bz}=\dfrac{1}{\xi_\bz (t)}\left[\begin{array}{ccc}
		0 & -\xi_\bz (t)\,\cos\varphi (t)& c\sin^2\varphi(t)\\
		 \xi_\bz (t)\,\cos\varphi(t)& 0 & x (t)\sin\varphi(t) \\
		-c\sin^2\varphi (t)& -x(t) \,\sin\varphi (t)& 0
	\end{array}\right], $$
	where
	\begin{eqnarray}\label{eq:sx}
		\bs^\bz(t,\theta)=\frac{1}{ \xi_\bz(t)} \left( -c\sin\varphi(t)\cos\varphi(t)\cos\theta+x(t)\sin\theta,-c\sin\varphi(t)\cos\varphi(t)\sin\theta-x(t)\cos\theta, \right. \nonumber\\ \left. -c\sin^2\varphi(t) \right)
	\end{eqnarray}
	and $\bn^\bz (t,\theta)$ is given by \eqref{eq:nuz}.
\end{prop}

\begin{proof}
	Since $\xi_\bz(t)\neq0$, for all $t$, it follows that $\bn^\bz (t, \theta)$, given by \eqref{eq:nuz}, is a unit normal vector field of $\bz$. Moreover, considering the unit vector field $\bs^\bz(t,\theta)$ as defined in \eqref{eq:sx}, we have  $(\bn^\bz\cdot \bs^\bz) (t, \theta)=0$, for all $(t,\te)$. Consequently,  $(\textbf{x},\bn^\bz,\bs^\bz)$ is a framed surface. 
	
	To compute the basic invariants, consider
	$$\bt^{\textbf{z}}(t,\theta)=(\bn^\bz\times \bs^\bz)(t, \theta) =(\sin\varphi(t)\cos\theta,\sin\varphi(t)\sin\theta,-\cos\varphi(t)).$$
Since  $\dot{\varphi}(t) = \ell (t)$,  the matrices $\mathcal{G}^{\textbf{x}}$, $\mathcal{F}^{\textbf{x}}_1$ and $\mathcal{F}_2^{\textbf{x}}$ are as stated above, as can be verified by direct calculation using the definitions given in \eqref{eq:g-framed-surface} and \eqref{eq:f-framed-surface}. \end{proof}

It is observed that when \( c = 0 \), the matrices coincide with those given  in \cite{TT}.

Given the basic invariants, we can compute the curvature \( C_{F}^{\bf z} = (J_F^\bz, K^\bz_F, H^\bz_F) \) of the framed surface,  which at $(t, \theta)$  is given by
\begin{multline}\label{eq:CFz}
C_{F}^{\bf z}= \left(-(\beta \xi_\bz)(t), \dfrac{(c^2\,\beta \sin^4 \varphi-  x^3 \ell  \cos\varphi)(t)}{\xi_\bz^3(t)}, \dfrac{[\beta (\xi^2_\bz + c^2\sin^2\varphi)\cos\varphi](t) + [x \ell (c^2 + x^2)](t)}{2\xi_\bz^2(t)}\right).
\end{multline}
We observe that the above curvature  does not depend on $\theta$.

By Proposition \ref{prop:FuTa-Fs}, \( (\bz, \bn^\bz) \) is a Legendre immersion around \( (t, \theta) \) if and only if \( C_{F}^\bz (t, \theta) \) does not vanish. According to \eqref{eq:CFz}, this condition fails if and only if \( \beta(t) = \ell(t) = 0 \) or \( \beta(t) = x(t) = 0 \), which means that either \( \gamma \) is not a front at \( t \) or it is singular at a point where it intersects the axis of \( \bz \). This result corresponds exactly to Remark \ref{rem:helifront}.

Now we compute the concomitant mapping of the framed surface \( (\bz, \bn^\bz, \bs^\bz) \),  which at $(t, \theta)$  is given by
\begin{align*}
	I_{F}^\bz=\bigg( & C_{F}^\bz (t),\, c (\ell \sin\varphi)(t),\,  -\dfrac{[(x \beta-c^2\ell\cos\varphi)\sin\varphi](t)}{\xi_\bz(t)}, \\ & \dfrac{c[(x \beta\sin^2\varphi+\ell (x^2+\xi_\bz)\cos\varphi)\sin\varphi](t)}{\xi_\bz^3(t)},\,(\ell \sin\varphi)(t), \, \dfrac{c(\beta \sin^2\varphi+x\ell \cos\varphi)(t)}{\xi_\bz(t)}\bigg),
\end{align*}
which also does not depend on $\theta$.

Thus, if \( c \neq 0 \), \( I_F^\bz \) vanishes at \( (t, \theta) \) if and only if  either \( \beta(t) = \ell(t) = 0 \) or $b(t) = x(t) = 0$. Therefore, according to Proposition \ref{prop:FuTa-Fs}, we have shown that if \( (\bz, \bn^\bz) \) is a Legendre immersion, then \( (\bz, \bn^\bz, \bs^\bz) \) is a framed immersion.

For completeness, we provide below the formulas for the helicoidal  surface \(\bx\) parametrized by \eqref{eq:heligenx}:
$$\bn^\bx(t,\theta)=\dfrac{((z\,\cos\varphi)(t),c\cos\varphi(t)\sin\theta+(z \sin\varphi)(t)\cos\theta,-c\cos\varphi(t)\cos\theta+(z\sin\varphi)(t)\sin\te)}{\xi_\bz(t)},$$

$$\bs^\bx(t,\theta)=\dfrac{(-c\cos^2\varphi(t),z(t)\sin\theta -c(\sin\varphi \cos\varphi)(t)\cos\theta,-c(\sin\varphi \cos\varphi)(t)\sin\theta-z(t)\cos\theta)}{\xi_\bx(t)},$$

$$\bt^{\textbf{x}}(t,\theta)=(\bn^\bx\times \bs^{\textbf{x}}) (t,\theta)=(-\sin\varphi(t),\cos\varphi(t)\cos\theta,\cos\varphi(t)\sin\theta).$$

In the following, we shall omit the parameter $t$ in the  functions $\xi_\bx, z, \beta, \varphi, \ell$, and $z$:

$$\mathcal{G}^{\textbf{x}}=\left[\begin{array}{cc}
	0 & \beta\\
	-\xi_\bx & -c \cos\varphi
\end{array}\right],$$
	$$\mathcal{F}_1^{\textbf{x}}=\dfrac{1}{\xi_\bx}\left[\begin{array}{ccc}
	0 &\frac{-c(\beta\cos^2\varphi+z \ell \sin\varphi)}{\xi_\bx}& z\,\ell\\
	\frac{c(\beta\cos^2\varphi+z \ell \sin\varphi)}{\xi_\bx} & 0 & -c \,\ell \cos\varphi\\
	-z\,\ell & c\, \ell  \cos\varphi & 0
\end{array}\right],$$
	$$\mathcal{F}_2^\bx=\dfrac{1}{\xi_\bx}\left[\begin{array}{ccc}
	0 & -\xi_\bx  \sin\varphi& c \cos^2\varphi\\
\xi_\bx  	\sin\varphi & 0 &  z\cos\varphi\\
	-c \cos^2\varphi & -z \cos\varphi & 0
\end{array}\right].$$

\begin{equation}\label{eq:CF}
	C_F^\bx=\left(\beta\xi_\bx,-\dfrac{c^2 \beta\, \cos^4\varphi- z^3 \ell \sin\varphi}{\xi_\bx^3},-\dfrac{\beta(\xi_\bx^2+c^2\cos^2\varphi)\sin\varphi +z\ell \, (c^2 + z^2)}{2\xi_\bx^2}\right).
\end{equation}
\begin{align*}
	I_{F}^\bx=\bigg( & C_{F}^\bx,-c\, \ell \cos\varphi,\frac{(\beta z-c^2\ell \sin\varphi)\cos\varphi}{\xi_\bx}, \\ & -\frac{c[\beta z\cos^2\varphi +\ell(\xi_\bx^2+z^2)\sin\varphi] \cos \varphi}{\xi_\bx^3},
	\ell \cos\varphi,- \frac{c(\beta\cos^2\varphi+z \ell \sin\varphi )}{\xi_\bx}
	\bigg).
\end{align*}

\section{Parallel surfaces}

Given a parametrization $\bx$  of a regular surface of revolution generated by a regular curve \(\gamma\), its parallel surface $\bx^\lambda$ is the surface of revolution generated by the parallel curve $\gamma^\lambda$ of \(\gamma\). The same conclusion extends to  surfaces of revolution generated by Legendre curves as demonstrated in Proposition 3.12 in \cite{TT}. However, the situation is different for helicoidal surfaces.

\begin{ex}
Considering the profile curve  \(\gamma(t) = (t, t)\), \(t > 0\), of the helicoidal surface \(\mathbf{z} (t, \theta) = (t \cos \theta, t \sin \theta, c\,\theta + t)\), its parallel curve is \(\gamma_\lambda (t) = (t - \frac{\lambda}{\sqrt{2}}, t + \frac{\lambda}{\sqrt{2}})\). The helicoidal surface around the \(z\)-axis generated by \(\gamma_\lambda\) is 
\[
 \left( (t + \frac{\lambda}{\sqrt{2}})\cos \theta, (t + \frac{\lambda}{\sqrt{2}})\sin \theta, c\,\theta + t + \frac{\lambda}{\sqrt{2}}\right).
\]
On the other hand, the parallel surface of $\bz$  is given by
	\[
	\mathbf{z}^\lambda (t, \theta) = \left( t(1 + \frac{\lambda}{\xi_\bz^\lambda (t)}) \cos \theta + \frac{\lambda c}{\xi_\bz^\lambda (t)}  \sin \theta, t (1 + \frac{\lambda }{\xi_\bz^\lambda (t)}) \sin \theta - \frac{\lambda  c}{\xi_\bz^\lambda (t)}  \cos \theta, c\,\theta + t (1 + \frac{\lambda}{\xi_\bz^\lambda (t)})\right),
	\]
	where \(\xi_\bz^\lambda (t) = \sqrt{c^2 + 2t^2}\).
	
	Thus, the helicoidal surface generated by a parallel curve of $\gamma$ is a parallel surface of the helicoidal surface generated by $\gamma$ if and only if $c=0$, that is, $\bz$ is a surface of revolution.

\end{ex}

The following result highlights the more intricate geometric behavior of parallel surfaces in the helicoidal setting, contrasting with the simpler case of surfaces of revolution.

\begin{theo}\label{theo:paralHelic}
	Let $(\gamma, \nu): I \rightarrow \R^2 \times S^1$  be a Legendre curve with curvature $(\ell,\beta)$, and let $\bz$ denote the  helicoidal surface generated by  $\gamma(t) = (x(t),z(t))$  around the $z$-axis  given by \eqref{eq:heligenz}, where  $x(t) \neq 0$, for all $t$.  Given $\lambda \in \R$, the parallel surface $\bz^\lambda: I \times \R \rightarrow \R^3$ of $\bz$  is  also a helicoidal surface, generated by a curve in the $xz$-plane which belongs to a 1-parameter deformation  of the parallel curve $\gamma_{\varepsilon \lambda}$ of  $\gamma$, where $\varepsilon = {\sgn({x(t)})}$.
\end{theo}

In order to prove this theorem, we first present the following lemma. 

We can consider helicoidal surfaces generated by a space curve $\alpha(t) = (x_1(t), x_2(t), x_3(t))$. For example, if the $z$-axis is the axis of the helicoidal surface,  it can be parametrized by:
\begin{equation}\label{eq:sh:espacial}
	{\bar \bz}(t, \theta) = (x_1(t) \cos \theta - x_2(t) \sin \theta, x_1(t) \sin \theta + x_2(t) \cos \theta, c\, \theta + x_3(t)) \ .
\end{equation}

For the local study of helicoidal surfaces, we can assume that the profile curve is a plane curve. More precisely, the following result shows that such a planar curve can be explicitly constructed, depending on the original spatial curve that generates the helicoidal surface.

\begin{lem}\label{lem:espacialtoplane}
The helicoidal surface given by \eqref{eq:sh:espacial} generated by a space curve is  a helicoidal surface generated by a plane curve.
\end{lem}

\begin{proof} 
Without loss of generality, we assume that   $\alpha(t) = (x_1(t), x_2(t), x_3(t))$   is defined in a neighborhood of  $t=0$.
		Let $m, n \in \mathbb{Z}$ be such that $x_1(t) = t^m \bar{x}_1(t)$ and $x_2(t) = t^n \bar{x}_2(t)$, where $\bar{x}_1(0)\neq 0$ and $\bar{x}_2(0)\neq 0$.  We shall consider two cases: $n \geq m$ and $n < m$.
	
	For the case $n \geq m$, we define the local diffeomorphism $h(t,\te) = (t,\te - \arctan (\frac{x_2(t)}{x_1(t)}) )$, which is well-defined near the origin, since $\frac{x_2(t)}{x_1(t)} = t^{n-m}\frac{\bar{x}_2(t)}{\bar{x}_1(t)}$ and $\bar{x}_1(0) \neq 0$.
	So 
	$$
	\bar{\bz}(h(t,\te)) = (X(t)\cos\te, X(t)\sin\te, Z(t) + c\,\te ),
	$$
	where
	\begin{equation}\label{eq:XZ1}
		\left\{\begin{array}{l}
	X (t)= \sgn (x_1(t)) \sqrt{x_1^2(t) + x_2^2(t)}\\ Z(t) = x_3(t) - c\,\arctan (\frac{x_2(t)}{x_1(t)}) . 
		\end{array}\right.
	\end{equation}

	On the other hand, if $n < m$, then the function $x_1(t)/x_2(t)$ is well-defined near the origin. So taking the diffeomorphism 
	$h(t,\te) = (t,\te + \arctan (\frac{x_1(t)}{x_2(t)}) - \frac{\pi}{2})$, we have
	$$
	\bar{\bz}(h(t,\te)) = (X(t)\cos\te, X(t)\sin\te, Z(t) + c\,\te ),
	$$
	where 
	\begin{equation}\label{eq:XZ2}
		\left\{\begin{array}{l}
		X(t) = \sgn (x_2(t))\sqrt{x_1^2(t) + x_2^2(t)} \\ Z(t) = x_3(t) + c (\arctan (\frac{x_1(t)}{x_2(t)} ) - \frac{\pi}{2}) .
		\end{array}\right.
	\end{equation}

	In both situation, $\bar{\bz} \circ h$ is a helicoidal surface around $z$-axis with slant $c$ and generated by the plane curve $(X(t), 0, Z(t))$.
\end{proof}

\medskip

\begin{rem}\label{rem:XZ}
	Let us observe that, if the space curve $\alpha(t) = (x_1(t), x_2(t), x_3(t))$ that generates the helicoidal surface $\bz$ satisfies \( x_1(t)x_2(t) \neq 0 \), then  \eqref{eq:XZ1} and \eqref{eq:XZ2} coincide if and only if \( x_1(t)x_2(t) > 0 \). In the case of $x_1(t)x_2(t) < 0$, the equations differ only by the isometry \( (x, z) \mapsto (-x, z + c \pi) \). This implies that, although the expressions may not coincide, they both represent intersections of \( \mathbf{z} \) with the \( xz \)-plane, but in different regions. Moreover,  any of them can be taken locally as a profile curve for \( \mathbf{z} \).
	
	Based on this observation, together with some analysis, we can express the profile plane curve $(X(t), 0,  Z(t))$ as follows:
	\[
	X(t) =
	\begin{cases}
		\operatorname{sgn}(x_2(t)) \sqrt{x_1^2(t) + x_2^2(t)}, & x_1(t)x_2(t) \neq 0;\\
		x_2(t), & x_1(t) = 0,\ x_2(t) \neq 0; \\
		x_1(t), & x_1(t) \neq 0,\ x_2(t) = 0; \\
		0, & x_1(t) = x_2(t) = 0.
	\end{cases}
	\]
	\[
	Z(t) =
	\begin{cases}
		x_3(t) - c\,\arctan\left( \dfrac{x_2(t)}{x_1(t)} \right), & x_1(t)x_2(t) > 0; \\
		x_3(t) - c\,\arctan\left( \dfrac{x_2(t)}{x_1(t)} \right) + c \pi, & x_1(t)x_2(t) < 0; \\
		x_3(t) - \dfrac{c\pi}{2}, & x_1(t) = 0,\ x_2(t) \neq 0; \\
		x_3(t), & x_2(t) = 0.
	\end{cases}
	\]
\end{rem}

\medskip

\begin{proof}\textbf{\hspace*{-.5cm} of Theorem \ref{theo:paralHelic}}
 \hspace*{.2cm}	
 Given $\lambda \in \R$, by direct calculations using \eqref{eq:heligenz} and \eqref{eq:nuz}, the  parallel surface of $\bz$ is given by:  
	\begin{equation*}
		\bz^\lambda (t, \theta)= (x_1(t) \cos \theta - x_2(t) \sin \theta, x_1(t) \sin \theta + x_2(t) \cos \theta, c\, \theta + x_3(t)),
	\end{equation*}
	where
		\begin{equation}\label{eq:zl} x_1(t) = x(t)(1  + \frac{\lambda \cos \varphi (t)}{\xi_\bz(t)}),\,  x_2 (t)= -\frac{c \,\lambda \sin \varphi (t)}{\xi_\bz(t)}, \, x_3(t)  = z(t) + \frac{\lambda \,x(t) \sin \varphi (t)}{\xi_\bz(t)}\, . 
				\end{equation}
	
	Then,  $\bz^\lambda$ is a helicoidal surface generated by the space curve $\alpha(t) = (x_1(t), x_2(t),$  $x_3(t))$. By Lemma \ref{lem:espacialtoplane}, this helicoidal is generated by a plane curve $\gamma_{\lambda,c} (t) =(X(t), 0,  Z(t))$, with $X,Z$  given by  Remark \ref{rem:XZ}.
	
	 When  $c = 0$, it holds that $\xi_\bz(t) = |x(t)|$. Since 
	 \begin{equation}\label{eq:gam:lambda:0}
\gamma_{\lambda,0}(t) = (x_1(t), 0, x_3(t)) = (x(t) + \varepsilon \lambda \cos \varphi (t), 0, z(t) + \varepsilon\lambda \sin \varphi(t))\,,
	 \end{equation}
	 where $\varepsilon = {\sgn({x(t)}})$,
	 then $\gamma_{\lambda,0}(t)  = \gamma(t) + \sgn(x(t)) \lambda \nu(t)$. 
	 Consequently, the curve $\gamma_{\lambda,0}$ coincides with the parallel curve $\gamma_\lambda$ (resp.  $\gamma_{-\lambda}$)  of $\gamma$  if $x(t) >0$ (resp. if $x(t) <0)$.
\end{proof}

In the following, we shall use the notation $\gamma_{\lambda,c}$ to denote both  the plane profile curve in the $xz$-plane  associated with the parallel surface $\bz^\lambda$, and the  1-parameter deformation of the parallel curve $\gamma_{\varepsilon \lambda}$ of $\gamma$  described  in Theorem \ref{theo:paralHelic}.

It is natural to ask whether, when $\gamma_{\varepsilon \lambda}$ is singular, the 1-parameter deformation $\gamma_{\lambda,c}$   is a versal deformation. The proposition below shows that this is not the case. 

\begin{prop}\label{lem:c=-cParalel}
	The deformation $\gamma_{\lambda,c}(t) = (X(t), 0, Z(t))$, with $X$ and $Z$ given in Remark \ref{rem:XZ}, depends only on the absolute value of $c$  up to isometry. That is, the curves $\gamma_{\lambda,c}$ and $\gamma_{\lambda,-c}$ are congruent via a rigid motion of $\mathbb{R}^3$.
\end{prop}
\begin{proof}
	By \eqref{eq:zl}, the functions $x_1$ and $x_3$  depend only on the absolute value of $c$, whereas $x_2$ changes its sign when $c$ changes sign. Thus, the result follows from the  expressions of $X$ and $Z$ given  in Remark \ref{rem:XZ}, and properties of the function $\arctan$. 
\end{proof}

The  symmetry given in the above proposition implies that, when $\gamma_{\varepsilon \lambda}$ is singular, the deformation does not  capture all possible local behaviors of the singularity. Hence, $\gamma_{\lambda,c}$ is not a versal deformation of $\gamma_{\varepsilon \lambda}$. This leads to a natural question: how is a singularity of  $\gamma_{\varepsilon \lambda}$  deformed by  $\gamma_{\lambda,c}$? This question is addressed in the next theorem.

By Lemma \ref{lem:espacialtoplane} and \eqref{eq:zl}, and assuming $x_1 (t) \neq 0$  for all $t$, we may take: 

\begin{equation}\label{eq:gammalambdac}
	\gamma_{\lambda,c} (t)= \left(\sgn (x_1(t))\sqrt{x_1^2(t) + x_2^2(t)}, x_3(t) - c \arctan (\frac{x_2(t)}{x_1(t)} )\right).
\end{equation}

\begin{rem} \label{rem:x1naoZero}
	Under the above notation, suppose that \( c = 0 \) and that the profile curve \( \gamma (t) =  (x(t), z(t))\) of \( \bz \) is regular at some \( t_0 \in I \), with $x(t_0)\neq 0$. If the parallel curve \( \gamma_{\varepsilon \lambda} \)  is singular at \( t_0 \), then \( x_1(t_0) \neq 0 \) if and only if \( \bz(t_0, \theta) \) is not an umbilic point of \( \bz \).
	In fact, since  $\gamma_{\varepsilon\lambda}= \gamma_{\lambda,0} = (x_1(t), x_3(t))$ given by \eqref{eq:gam:lambda:0},  then $\dot{\gamma}_{\varepsilon\lambda}(t) = (\beta(t) + \varepsilon \lambda \ell(t) )(-\sin \varphi(t), \cos \varphi (t))$, where $(\ell, \beta)$ is the curvature of the Legendre curve $(\gamma, \nu)$. Therefore,  \( \gamma_{\varepsilon\lambda} \) is singular at \( t_0 \) if and only if $ 
	\lambda = - \varepsilon\beta(t_0)/\ell(t_0)$.  By  \eqref{eq:gam:lambda:0},  
	$
	x_1(t) = x(t) + \varepsilon \lambda \cos\varphi(t).
	$
	Substituting \( \lambda\), we get
	\[
	x_1(t_0) = x(t_0) - \frac{\beta(t_0)}{\ell(t_0)} \cos\varphi(t_0).
	\]
	Therefore, \( x_1(t_0) = 0 \) if and only if
	$
	x(t_0)\ell(t_0) - \beta(t_0)\cos\varphi(t_0) = 0.
	$
	
	Now, observe that \( \bz \) is a regular surface of revolution in a neighborhood of \( (t_0, \theta) \), and its principal curvatures at that point are given by 
	$ -\frac{\ell(t_0)}{\beta(t_0)}$ and $-\frac{\cos\varphi(t_0)}{x(t_0)}.
	$
	Then, the point \( \bz(t_0, \theta) \) is umbilic if and only if 
	$$ x(t_0)\ell(t_0) - \beta(t_0)\cos\varphi(t_0) = 0.$$
	Thus, \( x_1(t_0) = 0 \) if and only if \( \bz(t_0, \theta) \) is an umbilic point, which proves the assertion.
	
	Moreover, if  \( x_1(t_0) \neq 0\), then \( \bz(t, \theta) \) is not an umbilic point  of the helicoidal surface $\bz$ for values of $c$ close to $0$ and of $t$ close to $t_0$.
\end{rem}

The following result shows that, under suitable geometric conditions, a singularity of the parallel curve persists in the one-parameter deformation associated to the helicoidal perturbation.

\begin{theo}\label{theo:def-par}
	Let $\bz$ be the helicoidal surface given by \eqref{eq:heligenz},  and  let $\gamma(t)= (x(t), z(t))$ denote its profile curve. Let $t_0 \in I$. Assume that at $t_0$,  the curve $\gamma$ is regular, has no vertex and $x(t_0)\neq 0$. Let $\gamma_{\lambda,c}$ be the 1-parameter deformation  of $\gamma_{\varepsilon\lambda}$ given in Theorem \ref{theo:paralHelic}, and let suppose that  $\gamma_{\varepsilon\lambda}$   is singular at $t_0$. Furthermore, assume that, when $c=0$, the point $\bz(t_0,\theta)$  is not  an umbilic of $\bz$.  Then  there exist neighborhoods $J$ of $0$ and  $I_0 \subset I$ of $t_0$ such that $\gamma_{\lambda,c}$ is singular at some $t \in I_0$ for all $c \in J$. 
\end{theo}

	\begin{proof}
		Since the point $\bz(t_0,\theta)$ is not an umbilic of $\bz$ when $c = 0$, then, by Remark~\ref{rem:x1naoZero}, $x_1(t_0) \neq 0$. Consequently, $\gamma_{\lambda,c}$ is given by \eqref{eq:gammalambdac}. We note that $\gamma_{\lambda,c}$ is singular at $t$ if and only if $\lambda$, $c$, and $t$ satisfy the following system:
		\[
		\left\{
		\begin{aligned}
			x_1 \dot{x}_1 + x_2 \dot{x}_2 &= 0 \\
			(x_1^2 + x_2^2)\dot{x}_3 - c(x_1 \dot{x}_2 - \dot{x}_1 x_2) &= 0,
		\end{aligned}
		\right.
		\]
		where $x_1$, $x_2$, and $x_3$ are given in \eqref{eq:zl}, and we omit the variable $t$.
		
		By isolating $\dot{x}_1$ in the first equation and substituting it into the second, the system above is equivalent to:
		\begin{equation}\label{eq:siteq}
			\left\{
			\begin{aligned}
				x_1 \dot{x}_1 + x_2 \dot{x}_2 &= 0 \\
				x_1 \dot{x}_3 - c\, \dot{x}_2 &= 0.
			\end{aligned}
			\right.
		\end{equation}
		
		From direct calculations, we conclude that the system above is equivalent to:
		\[
		\left\{
		\begin{aligned}
			x(t)\, \Phi(c,t) \sin\varphi(t) &= 0 \\
			x(t)\, \Phi(c,t) \cos\varphi(t) &= 0,
		\end{aligned}
		\right.
		\]
		where
		\begin{multline}\label{eq:phi}
			\Phi(c,t) = \beta (t) [c^2(c^2 - \lambda^2)\sin^4\varphi(t) + (2\xi_\bz^2(t) - x^2(t))(x^2(t) + \lambda \xi_\bz(t) \cos\varphi(t))] +\\
			\lambda x(t) \ell(t) [(c^2 + x^2(t))\xi_\bz (t)+ \lambda x^2(t) \cos\varphi(t)].
		\end{multline}
		
		Since $x(t_0)\neq 0$, there exists a neighborhood $\bar{I} \subset I$ of $t_0$ such that $x(t) \neq 0$ for all $t \in \bar{I}$. Thus, $\gamma_{\lambda,c}$ is singular at $t \in \bar{I}$ if and only if $\Phi(c,t) = 0$.
		
		Since $\dot{\gamma}_{\varepsilon\lambda}(t_0) = 0$, then
		\(
		\lambda = -\varepsilon \frac{\beta(t_0)}{\ell(t_0)}
		\) (see Remark \ref{rem:x1naoZero}).
		Therefore,
		\[
		\frac{\partial \Phi}{\partial t}(0,t_0) = \frac{x^3(t_0)}{\ell^2(t_0)} 
		\big[x(t_0)\ell(t_0) - \beta(t_0)\cos\varphi(t_0)\big] 
		\big[\dot{\beta}(t_0)\ell(t_0) - \beta(t_0)\dot{\ell}(t_0)\big] \neq 0,
		\]
		since $x(t_0)\ell(t_0) - \beta(t_0)\cos\varphi(t_0) \neq 0$ (as $x_1(t_0) \neq 0$, see Remark~\ref{rem:Lem}), and $\dot{\beta}(t_0) \ell(t_0) - \beta(t_0) \dot{\ell}(t_0)= 0$ if and only if $\dot{\kappa}(t_0) = 0$ (since $\kappa(t)=\ell(t)/|\beta(t)|$), that is, if and only if $\gamma$ has a vertex at $t_0$.
		
		Then, by the implicit function theorem, there exist neighborhoods $J$ of $0$ and $I_0 \subset \bar{I}$ of $t_0$, and a smooth function $t(c)$ with $t(0) = t_0$ such that $\Phi(c,t(c)) = 0$ for all $c \in J$ and $t(c) \in I_0$. In other words, for each $c \in J$, there exists a $t_1 \in I_0$ such that $\gamma_{\lambda,c}$ is singular at $t_1$.
	\end{proof}

We illustrate this theorem with the following example:

\begin{ex}
	Let \(\gamma(t) = (t+2, t^2/2)\) and \(\lambda = -2\sqrt{2}\). Then \(\dot{\gamma_\lambda}(t) = 0\) for \(t = \pm 1\), and the deformation \(\gamma_{\lambda,c}\) is depicted in Figure \ref{fig:ex_par} by the curves with \(c = 0, 1,\) and \(1.5\).
\end{ex}

\begin{figure}[h]
	\centering 
	\includegraphics[width=6cm]{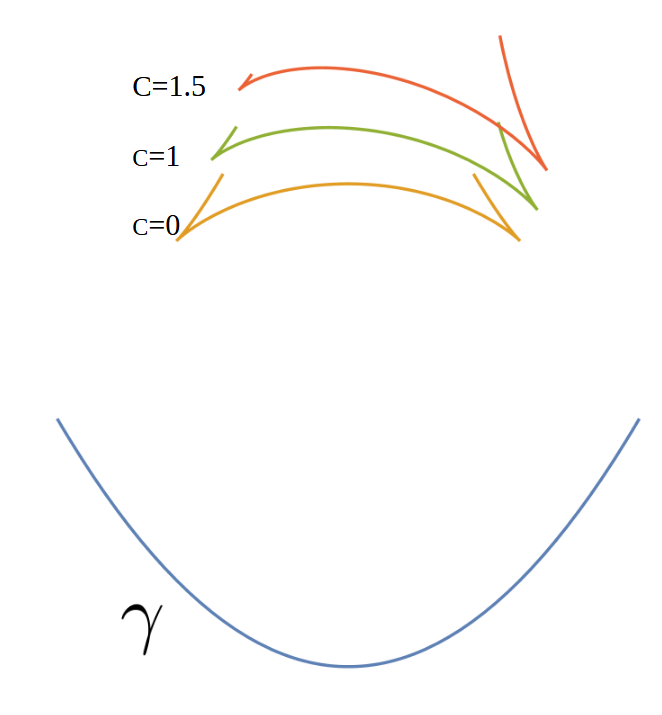}
	\caption{Deformation  $\gamma_{\lambda,c}$, for $\gamma=(t+2,\frac{t^2}{2})$ and $\lambda=-2\sqrt{2}$.}
	\label{fig:ex_par}
\end{figure}

\section{Focal surfaces}
We now turn our attention to the   focal surfaces associated with a helicoidal surface, showing that it occors a different situation if compared with what happens  for surfaces of revolution generated by  a frontal(see Proposition 3.13  of  \cite{TT}). 

As before, we focus on helicoidal surfaces generated by a curve around the $z$-axis. In this setting, and in accordance with the definition of the evolute of a curve, we assume that the profile curve is a front. Let \((\gamma, \nu): I \to \mathbb{R}^2 \times S^1\) be a Legendre immersion with curvature \((\ell, \beta)\), where \(\gamma(t) = (x(t), z(t))\), and let \(\bz\) denote the helicoidal surface generated by \(\gamma\) around the $z$-axis given by \eqref{eq:heligenz}. We further suppose that \(\xi_\bz(t) \neq 0\) for all \(t \in I\).  Recall that a focal surface of \eqref{eq:heligenz}  is given by
$$\mathcal{E}(\bz)(t,\theta) = \textbf{z}(t,\theta) + \lambda(t,\theta)\,\bn^\bz(t,\theta),$$
where $\bn^\bz$ is given by \eqref{eq:nuz} and   \(\lambda(t,\theta)\) is a solution of \eqref{eq:solutionFS}. Since $J_F^\bz,  K^\bz_F$ and  $H^\bz_F$ are functions depending only on $t$ (as seen in  $\eqref{eq:CFz}$), it follows that  $\lambda$ is a function of $t$ alone. Thus,
\[
\mathcal{E}(\bz)(t,\theta) = (\bar{x}(t)\cos\theta - \bar{y}(t)\sin\theta, \bar{x}(t)\sin\theta + \bar{y}(t)\cos\theta, c\,\theta + \bar{z}(t)),
\]
where 
\begin{equation}\label{eq:xi}
	\bar{x}(t)= x(t)(1+ \frac{\lambda(t) \cos\varphi(t)}{\xi_\bz(t)}),\ \ \ \bar{y}(t)=-\dfrac{c\, \lambda(t)\,\sin\varphi(t)}{\xi_\bz(t)}, \ \ \bar{z}(t) =z(t) +\dfrac{\lambda(t)\, x(t)\sin\varphi(t)}{\xi_\bz(t)}.
\end{equation}

These identities coincide with those in \eqref{eq:zl}, except for the key distinction that, in this context, \(\lambda\) is not a constant but a function of \(t\). 
This analysis reveals that the focal surfaces \(\mathcal{E}(\bz)\) preserve the  helicoidal structure of the original surface, with their profile curves described by a space curve.
This geometric property can be formalized in the following proposition:

\begin{prop}\label{prop:focal:gen}
	Let $\bz: I \times \R \to \mathbb{R}^3$ be a helicoidal surface around the $z$-axis given by \eqref{eq:heligenz}  generated by a Legendre immersion $(\gamma, \nu)$, with $\gamma(t) = (x(t), z(t))$, and suppose that $\xi_\bz(t) \neq 0$ for all $t \in I$. Then, the focal surface $\mathcal{E}(\bz)$ associated to $\bz$ is also a helicoidal surface about the $z$-axis, whose profile curve is the space curve  $(\bar{x}(t), \bar{y}(t), \bar{z}(t))$ given by \eqref{eq:xi}.
\end{prop}

\begin{rem}\label{rem:Lem}
	Under the notation of the above proposition, let $(\ell, \beta)$ be the curvature of the Legendre curve $(\gamma, \nu)$. According to the proof of Lemma~\ref{lem:espacialtoplane}, the focal surface $\mathcal{E}(\bz)$ has the following planar profile curves:
	\begin{equation}\label{eq:alphai}
		\delta_c(t) = \left( \sgn(\bar{x}(t)) \sqrt{\bar{x}(t)^2 + \bar{y}(t)^2},\, \bar{z}(t) - c\, \arctan\left( \dfrac{\bar{y}(t)}{\bar{x}(t)} \right) \right),
	\end{equation}
	or
	\[
	\left( \sgn(\bar{y}(t)) \sqrt{\bar{x}(t)^2 + \bar{y}(t)^2},\, \bar{z}(t) + c\, \arctan\left( \dfrac{\bar{x}(t)}{\bar{y}(t)} \right) - \dfrac{c\pi}{2} \right),
	\]
	where $\bar{x}(t)$, $\bar{y}(t)$, and $\bar{z}(t)$ are given in \eqref{eq:xi}.
	
	When \(c = 0\), the values of \(\lambda\) that satisfy \eqref{eq:solutionFS} are \(\lambda_1(t) = -\varepsilon\, \dfrac{\beta(t)}{\ell(t)}\) and \(\lambda_2(t) = -\varepsilon\, \dfrac{x(t)}{\cos \varphi(t)}\), where \(\varepsilon = \sgn(x(t))\) (note that \(\xi_\bz(t) \neq 0\) implies \(x(t) \neq 0\)).
	
	Thus, for \(\lambda = \lambda_1\), the curve \(\delta_0(t)\) coincides with the evolute \(\mathcal{E}_\gamma(t)\) of \(\gamma\). Consequently, \(\delta_c(t)\) defines a one-parameter deformation of the evolute of \(\gamma\).
	
	For \(\lambda = \lambda_2\), the second planar profile curve given above corresponds to a deformation of a curve lying on the axis of revolution, which is not of interest in our context.
\end{rem}

\begin{ex}\label{ex:focalcirc}
	Let \(\bz\) be the helicoidal surface with profile curve given by $\gamma(t)=(\sin t ,-\cos t + 2)$ and with slant \(c = 2\) as illustrated in Figure \ref{fig:focalcirc}, left. By the proof of Lemma \ref{lem:espacialtoplane}, we determine  the plane profile  curve associated with the focal surfaces $\mathcal{E}(\bz)$ of \(\bz\). One of them  is represented by the green quadricusp shown in Figure \ref{fig:focalcirc}, rigth. The blue curve represents \(\gamma\). 
\end{ex}

\begin{figure}[h]
	\centering
	\begin{subfigure}{0.4\textwidth}
		\includegraphics[width=\textwidth]{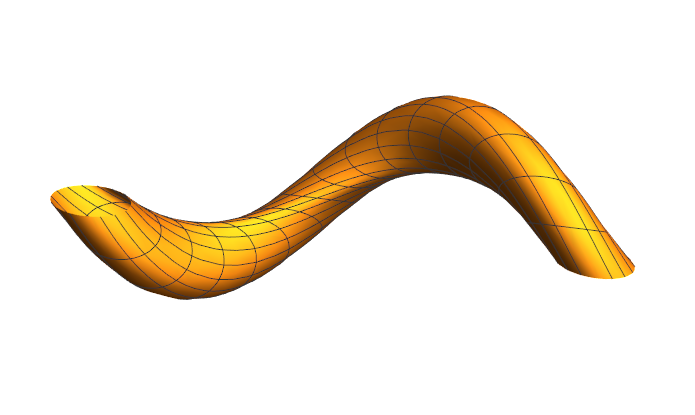}

	\end{subfigure}
	\hfill
	\begin{subfigure}{0.4\textwidth}
		\includegraphics[width=\textwidth]{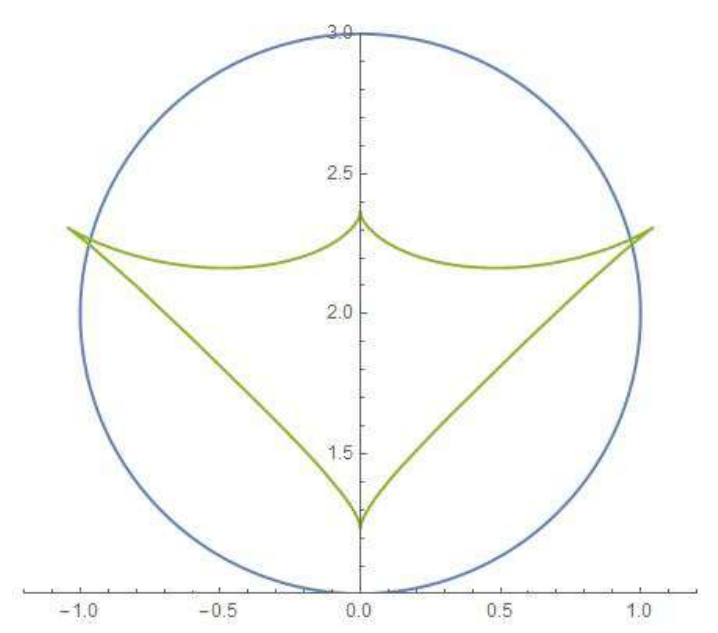}
		\label{fig:second}
	\end{subfigure}	
\caption{Figures of the Example \ref{ex:focalcirc}.}
		\label{fig:focalcirc}
\end{figure}

In the following, we shall consider $\gamma$ and $\mathcal{E}(\bz)$ locally at $t_0$ and assume that both of them do not intersect the axis of $\bz$ ($x(t_0)\neq0$ and $\bar{x}(t_0)\neq0$). 
 We aim to study the behavior of 
the deformation  $\delta_c(t)$ of the evolute $\mathcal{E}_\gamma(t)=\delta_0(t)$ of $\gamma$.

\begin{rem}
	Assume $c=0$ and $x(t_0)\neq 0$. Let $\bar{x}(t)$ be given by \eqref{eq:xi}. Then \(\bar{x}(t_0)\neq0\) if and only if $\bz(t_0,\theta)$ is not an umbilic point of $\bz$ (see Remark \ref{rem:x1naoZero}). 
\end{rem}

\begin{ex}
	Let \(\gamma\) be the curve from Example \ref{ex:focalcirc}. Its evolute is an isolated point, which is deformed by \(\delta(c,t)\) into a quadricusp, as shown in Figure \ref{fig:focalcirc}.
\end{ex}

As we did for the deformation \(\gamma_{\lambda,c}\) of the parallel curve \(\gamma_\lambda\) of the profile curve \(\gamma\), we observe that the deformation \(\delta_c(t)\) is also not a versal deformation. Similarly to Proposition~\ref{lem:c=-cParalel}, we have the following result:

\begin{prop}\label{lem:c=-c}
	The deformation in \eqref{eq:alphai} of $\gamma$ depends only on the absolute value of the parameter \(c\), up to isometry.
\end{prop}

We now investigate how a singularity of the evolute $\mathcal{E}_\gamma$ of $\gamma$ is deformed by the family $\delta_c(t)$.

\begin{theo}\label{theo:def-foc}
Let $\bz$ be the helicoidal surface given by \eqref{eq:heligenz},  and  let $\gamma(t)= (x(t), z(t))$ denote its profile curve. Let $t_0 \in I$. Assume that at $t_0$,  the curve $\gamma$ is regular, has an ordinary vertex and $x(t_0)\neq 0$. 
	Let  \(\delta_c(t)\) be the deformation of the evolute \(\mathcal{E}_\gamma \) of $\gamma$ given by \eqref{eq:alphai}.  Assume that the curvature $K_F^\bz$  given in \eqref{eq:CFz} does not vanish when  $c=0$.
	Then  there exist neighborhoods $J$ of $0$ and  $I_0 \subset I$ of $t_0$ such that $\delta_{c}$ is singular at some $t \in I_0$ for all $c \in J$.  
\end{theo}

\begin{proof}
	Similarly to the proof of Theorem \ref{theo:def-par}, the curve $\delta_c$ is singular at $t_0$ if and only if the following identities are satisfied at $t_0$:
	
	\begin{equation*}\label{eq:defev}
		\left\{\begin{array}{lr}
			\bar{x} \dot{\bar{x}}+\bar{y}\dot{\bar{y}}=0\, \\
			\dot{\bar{z}}\bar{x}-c\,\dot{\bar{y}}=0.\\	
		\end{array}\right.
	\end{equation*}
	
	We will prove the existence of a smooth function $t=t(c)$, defined in a neighborhood of $c=0$, such that $t(0)=t_0$ and satisfying the above system. This guarantees that, for each small value of $c$, the deformed curve $\delta_c(t)$ continues to exhibit a singularity at some point near $t_0$, confirming the persistence of the singularity under deformation.
	
	From direct calculations, we conclude that the above system is equivalent to: 
	
	\begin{equation*}\label{eq:defev2}
		\left\{\begin{array}{ll}
			\Phi(c,t) +  \dot{\lambda}_c(t) \Delta_1(c,t)
			=0\, \\
			\Phi(c,t) +  \dot{\lambda}_c(t)	\Delta_2(c,t)=0,\\	
		\end{array}\right.
	\end{equation*}
	where \(\Delta_1 = \left( \lambda \, \xi_\bz^2 + x^2 \left( -\lambda \sin^2 \varphi + \xi_\bz \cos \varphi \right) \right)\), \(\Delta_2 = \sin \varphi \left( c^2 \xi_\bz + x^2 \left( \lambda \cos \varphi + \xi_\bz \right) \right)\), and $\Phi$ is given by \eqref{eq:phi}. Here, we denote $\lambda$ by $\lambda_c$ to emphasize its dependence on the parameter $c$ as well. With some calculations, we can simplify $\Phi$ as
	\begin{eqnarray}
		\Phi(c,t) = - [ K_F^\bz(t)\lambda_c^2(t) - 2H_F^\bz(t)\lambda_c(t) + J_F^\bz(t)]\, \xi_\bz^3(t)\,,
	\end{eqnarray}

	Since  \(\lambda_c(t)\) satisfies \eqref{eq:solutionFS}, then $\Phi(c,t)=0$, and the zeros of \eqref{eq:defev2} are given by the union of \(\dot{\lambda}_c(t) =0\) with the set \(\Delta_1^{-1}(0) \cap \Delta_2^{-1}(0)\).
	However, we have 
	$$
	\Delta_1(0,t_0)=\cos\varphi(t_0)x(t_0)^2(x(t_0)+\lambda(t_0)\cos\varphi(t_0))=\cos\varphi(t_0)x(t_0)^2\bar{x}(t_0)$$ and
	$$\Delta_2(0,t_0)=\sin\varphi(t_0)x(t_0)^3(x(t_0)+\lambda(t_0)\cos\varphi(t_0))=\sin\varphi(t_0)x(t_0)^3\bar{x}(t_0).
	$$
	Thus, \(\Delta_1^{-1}(0) \cap \Delta_2^{-1}(0)\) is empty near $(0,t_0)$ by hypothesis.
	
	On the other hand, by \eqref{eq:solutionFS}, the function \(\lambda_c\) is given by \(\lambda_c^- = \frac{H_F^\bz - \sqrt{(H_F^\bz)^2 - J_F^\bz K_F^\bz}}{K_F^\bz}\) or\linebreak \(\lambda_c^+ = \frac{H_F^\bz + \sqrt{(H_F^\bz)^2 - J_F^\bz K_F^\bz}}{K_F^\bz}\). With some calculation, we obtain that $\lambda_c^{-\sgn(\bar{x}(t_0))}$ is the one that gives the deformation \eqref{eq:alphai} of $\mathcal{E}(\bz)$. For simplicity, in the following, we shall denote it only by $\lambda_c$.
	
	Thus, we obtain
	\[
	\frac{\partial}{\partial t}\dot{\lambda}_0(t_0) = \frac{-\ell^2(t_0) \ddot{\beta}(t_0) - 2 \beta(t_0) \dot{\ell}^2(t_0) + \ell(t_0) \left( 2 \dot{\beta}(t_0) \dot{\ell}(t_0) + \beta(t_0) \ddot{\ell}(t_0) \right)}{\ell^2(t_0)},
	\]
where $(\ell, \beta)$ is the curvature of  $(\gamma, \nu)$.	
	
	Since $t_0$ is a vertex of $\gamma$, then  $\dot \kappa (t_0) = 0$, that is, \(\dot{\beta}(t_0) \, \ell(t_0) - \beta(t_0) \, \dot \ell(t_0) = 0\) holds, which allows us to further simplify the identity above to:

	\[
	\Phi_t(0,t_0) = \frac{-\ell(t_0) \, \ddot{\beta}(t_0) + \beta(t_0) \, \ddot{\ell}(t_0)}{\ell^2(t_0)},
	\]
	which vanishes if and only if \(\gamma\) has a vertex of at least order two. Thus, by assumption, $\frac{\partial}{\partial t}\dot{\lambda}_0(t_0) \neq 0$, and by the implicit function theorem, there exists a unique smooth curve \((c,t(c))\) passing through the origin such that \(\dot{\lambda}_c(t(c)) = 0\) for all \(c\) near the origin.
\end{proof}

\section{Behavior of Gaussian and Mean curvatures}

This section investigates the behavior of the Gaussian and mean curvatures of helicoidal surfaces, with particular attention to whether these curvatures remain bounded or diverge near singular points. We shall use the notation introduced in the previous sections.
	Let $(\gamma, \nu): I \rightarrow \R^2 \times S^1$  be a Legendre curve with curvature $(\ell,\beta)$, and let $\bz$ denote the  helicoidal surface generated by  $\gamma(t) = (x(t),z(t))$  around the $z$-axis  given by \eqref{eq:heligenz}. 
	 Without loss of generality, we assume that $\gamma$ has an isolated singularity at the origin. This implies that $\bz$ has a curve of singular points given by $\{(0, \theta), \, \theta \in \R\}$. Since the signed area density function $\Lambda = \det(\bz_t, \bz_\theta, \bn_\bz) = -\xi_\bz \beta$ does not depend on $\theta$, we have $\Lambda_t(0, \theta) = -\xi_\bz(0) \dot\beta(0)$. Consequently, the singularities of $\bz$ are  non-degenerate if and only if $\dot \beta(0) \neq 0$. 

We express the function $\beta$ as $\beta(t) = t^m \bar{\beta}(t)$, where $\bar{\beta}(t)$ is smooth and satisfies $\bar{\beta}(0) \neq 0$. Therefore, $(0, \theta)$ is a  non-degenerate singular point of $\bz$ if and only if $m = 1$. It is worth recalling that $\gamma$ is a front at $t_0$ if and only if $(\ell, \beta) \neq (0, 0)$ at $t_0$.

We begin by analyzing the Gaussian curvature \(K\) near the singularities of \(\mathbf{z}\). The Gaussian curvature can be written as

\[
K(t)=\dfrac{K_F^\bz(t)}{J_F^\bz(t)} = -\dfrac{c^2 t^m \bar{\beta}(t) \sin^4 \varphi(t) - x^3(t) \ell(t) \cos \varphi(t)}{t^m \bar{\beta}(t) \xi_\mathbf{z}^4(t)}.
\]
It follows that \(K\) is bounded if and only if the function
\[
f(t) =  \ell(t)h(t),
\]
where $h(t)=x^3(t)\cos\varphi(t)$, satifies  \(f(t) = O(t^m)\), meaning that the quotient \(\frac{f(t)}{t^m}\) remains bounded near the origin.

We now analyze this condition by distinguishing between degenerate and non-degenerate singularities, with further subdivisions based on whether the generating curve \(\gamma\) is a front or not:

\begin{itemize}
	\item[(1)] Non-degenerate singularities (\(m = 1\)).

	\begin{itemize}
		\item \(\gamma\) is a front. Since \(\ell(0) \neq 0\), \(K\) is bounded if and only if \(h(t) = O(t)\), which occurs precisely when \(x(0) = 0\) or \(\cos \varphi(0) = 0\). In geometric terms, the curvature remains bounded near the singular curve if and only if the generating curve intersects the axis of rotation at the singularity or if its normal vector is parallel to the axis of rotation.
		\item \(\gamma\) is not a front. Here, \(\ell(0) = 0\), and therefore the term \(f(t)=0\), ensuring that \(K\) is bounded.
	\end{itemize}
	\item[(2)] Degenerate singularities (\(m > 1\)).
	\begin{itemize}
		\item \(\gamma\) is a front. Since \(\ell(0) \neq 0\), \(K\) is bounded if and only if $h(t)= O(t^m)$.
		Noting that \(\dot x(t) = -\beta(t) \sin \varphi(t) = O(t^m)\), we conclude that if \(x(0) = 0\), then \(x(t) = O(t^{m+1})\), which ensures \(K\) is bounded. On the other hand, if \(x(0) \neq 0\) and \(\cos \varphi(0) = 0\), then
		$
		\dot h(0) = -\sin \varphi(0)  x^3(0) \ell(0) \neq 0,
		$
		implying that \(h(t)\) can not satisfy \(h(t)=O(t^m)\), and hence \(K\) is unbounded. The case $x(0)\neq0$ and $\cos\varphi(0)\neq0$ implies directly that $K$ is unbounded. Therefore, in this case, the curvature is bounded if and only if \(x(0) = 0\).
		\item \(\gamma\) is not a front. In this situation, \(\ell(0) = 0\), and the condition \(K\) bounded becomes equivalent to
	$
		f(t)= O(t^m).
$

		If \(x(0) = 0\), then as before \(x(t) = O(t^{m+1})\), and \(K\) is bounded. Now suppose that \(x(0) \cos \varphi(0) \neq 0\). Then, \(K\) is bounded if and only if \(\ell(t) = O(t^m)\), which implies \(\varphi(t) = O(t^{m+1})\).
		
		If $x(0)\neq0$ and \(\cos \varphi(0) = 0\), direct computation shows that \(\varphi(t) = O(t^n)\) if and only if \(f(t) = O(t^{2n-1})\). Thus, \(K\) is bounded if and only if
		\[
		\varphi(t) = 
		\begin{cases}
			O(t^{\frac{m+1}{2}}), & \text{if } m \text{ is odd}, \\
			O(t^{\frac{m}{2}}), & \text{if } m \text{ is even}.
		\end{cases}
		\]
	\end{itemize}
\end{itemize}
The conditions under which the Gaussian curvature remains bounded are summarized in Table~\ref{tab:K}.

\begin{table}[h!]
	\centering
	\renewcommand{\arraystretch}{1.3}
	\begin{tabular}{|>{\centering\arraybackslash}m{0.18\linewidth}|
			>{\centering\arraybackslash}m{0.11\linewidth}|
			>{\centering\arraybackslash}m{0.18\linewidth}|
			>{\centering\arraybackslash}m{0.38\linewidth}|}
		\hline
		\textbf{Singularities} & \(\gamma\) & \textbf{Condition} & \textbf{Behavior} \\ \hline
		
		\multirow{4}{*}{\shortstack{Non-degenerate \\ \((m = 1)\)}} 
		& \multirow{3}{*}{Front}
		& \multirow{2}{*}{\(\cos \varphi(0)x(0) = 0\)} & \multirow{2}{*}{Bounded} \\
		&  &  &  \\ \cline{3-4}
		&  &\(\cos \varphi(0)x(0) \neq 0\)& Unbounded \\ \cline{2-4}
		& Non-front & — & Bounded \\ \hline
		
		\multirow{5}{*}{\shortstack{Degenerate \\ \((m > 1)\)}} 
		& \multirow{2}{*}{Front} 
		& \(x = 0\) & Bounded \\ \cline{3-4}
		&  & \(x \neq 0\) & Unbounded \\ \cline{2-4}
		& \multirow{3}{*}{Non-front} 
		& \(x = 0\) & Bounded \\ \cline{3-4}
		&  & \(x \cos \varphi \neq 0\) & Bounded \(\Leftrightarrow \varphi(t) = O(t^{m+1})\) \\ \cline{3-4}
		&  & \(x \neq 0\), \(\cos \varphi = 0\) & 
		\parbox[c]{\linewidth}{\centering
			\rule{0pt}{4.5ex}\raisebox{1ex}{Bounded $\Leftrightarrow$} \( \left\{
			\begin{array}{ll}
				\varphi(t) = O(t^{\frac{m+1}{2}}), & \text{if \(m\) is odd} \\[4pt]
				\varphi(t) = O(t^{\frac{m}{2}}), & \text{if \(m\) is even}
			\end{array}
			\right.\)} 
		\\ \hline
	\end{tabular}
	\caption{Summary of the behavior of the Gaussian curvature.}
	\label{tab:K}
\end{table}

Now we analyze the behavior of the mean curvature \( H \) near the singular points of a helicoidal surface. From the expression  
\[
H(t)=\dfrac{H_F^\bz(t)}{J_F^\bz(t)} = -\dfrac{t^m \bar{\beta}(t) \cos \varphi(t) \left( \xi^2(t) + c^2 \sin^2 \varphi(t) \right) + x(t) \ell(t) (x^2(t) + c^2)}{2 t^m \bar{\beta}(t) \xi^3(t)},
\]
we conclude that \( H \) is bounded if and only if \( x(t) \ell(t) (x^2(t) + c^2) = O(t^m) \). We consider the followig cases:

\begin{itemize}
	\item[(1)] Non-degenerate singularities (\( m = 1 \)):
	\begin{itemize}
		\item \(\gamma\) is a front. In this case,  \(\ell(0) \neq 0\). Thus, \( H \) is bounded if and only if \( x(0) = 0 \).
		\item \(\gamma\) is not a front. In this case, \(\ell(0) = 0\), and therefore \( H \) is bounded.
	\end{itemize}
	\item[(2)] Degenerate singularities (\( m > 1 \)):
	\begin{itemize}
		\item \(\gamma\) is a front. Then \(\ell(0) \neq 0\). In this case, \( H \) is bounded if and only if \( x(x^2 + c^2) = O(t^m) \). We conclude, in the same way for the Gaussian curvature, that \( H \) is bounded if and only if \( x(0) = 0 \).
		\item \(\gamma\) is not a front. Then \(\ell(0) = 0 \). If \( x(0) = 0 \), then \( H \) is bounded for any integer \( m \). If \( x(0) \neq 0 \), we deduce that \( H \) is bounded if and only if \( \ell(t)=\dot \varphi(t) = O(t^m) \), or equivalently, \( \varphi(t) = O(t^{m+1}) \).
	\end{itemize}
\end{itemize}

The results obtained above for the mean curvature \( H \) are summarized in Table~\ref{tab:K2}.

\begin{table}[h!]
	\centering
	\begin{tabular}{|c|c|c|c|}
		\hline
		\textbf{Singularities} & \textbf{\(\gamma\)} & \textbf{Condition} & \textbf{Behavior of \(H(t)\)} \\ \hline
		
		\multirow{4}{*}{\shortstack{Non-degenerate \\ \((m = 1)\)}} 
		& \multirow{2}{*}{Front}       & \(x(0) = 0\)   & \(H\) is bounded   \\ \cline{3-4}
		&                              & \(x(0) \neq 0\)& \(H\) is unbounded \\ \cline{2-4}
		& Not a front                  & —              & \(H\) is bounded   \\ \hline
		
		\multirow{5}{*}{\shortstack{Degenerate \\ \((m > 1)\)}} 
		& \multirow{2}{*}{Front}       & \(x(0) = 0\)   & \(H\) is bounded   \\ \cline{3-4}
		&                              & \(x(0) \neq 0\)& \(H\) is unbounded \\ \cline{2-4}
		& \multirow{2}{*}{Not a front} & \(x(0) = 0\)   & \(H\) is bounded   \\ \cline{3-4}
		&                              & \(x(0) \neq 0\)& \(H\) is bounded \(\Leftrightarrow \varphi(t) = O(t^{m+1})\) \\ \hline
		
	\end{tabular}
	\caption{Summary of the conditions for the boundedness of the mean curvature \(H\) near singular points.}
	\label{tab:K2}
\end{table}

We conclude that, when  $x(0) = 0$, the curvatures $K$ and \( H \) are always bounded near the singular points of the surface.

\begin{rem}
	When the profile curve $\gamma(t) = (x(t), z(t))$ of $\bz$ does not intersect the axis of $\bz$ ($x(t) \neq 0$) and has an isolated singularity, the surface $\bz$ is a $\sigma$-edge (see \cite{sigmaedge} for more details). In fact, the map 
	$$
	\Psi(X,Y,Z) = (X \cos(Y), X \sin(Y), Z + c\, Y)
	$$
takes the surface $(x(t), \theta, z(t))$ to the surface $\bz$, and it is a diffeomorphism since\linebreak $\det(Jac(\Psi(x(t),0,z(t))) = x(t) \neq 0$.
	
Therefore, the classical  geometric invariants for frontals, namely, the singular, normal and  cuspidal curvatures, as well as the cuspidal torsion) (see \cite{martins-saji,suy}) which have been extended to \(\sigma\)-edges in \cite{sigmaedge}, are also applicable to helicoidal surfaces.
\end{rem}

\medskip
\begin{tabular}{ll}
	\begin{tabular}{l}
		Dep. de Matem{\'a}tica, IBILCE - UNESP \\
		R. Crist{\'o}v{\~a}o Colombo, 2265,  CEP 15054-000\\
		S{\~a}o Jos{\'e} do Rio Preto, SP, Brazil\\
		E-mail: {\tt luciana.martins@unesp.br}
	\end{tabular}
	&
	\begin{tabular}{l}
		Dep. de Matem{\'a}tica, ICMC - USP \\
Caixa postal 668, CEP 13560-970\\
		S{\~a}o Carlos, SP, Brazil\\
		E-mail: {\tt samuelp.santos@hotmail.com}
	\end{tabular}
\end{tabular}


\begin{thebibliography}{99}

\bibitem{Ball} R. S. Ball, {\it A Treatise on the Theory of Screws}, reprint of the 1876 ed. Cambridge, U.K.: Cambridge University Press, 1998.

\bibitem{FuTa-lc1}   T. Fukunaga, M. Takahashi, Existence and uniqueness for Legendre curves.  \textit{J. Geom.} {\bf 104}  (2013), 297--307.

\bibitem{FuTa-lc2}   T. Fukunaga, M. Takahashi, Evolutes of fronts in the Euclidean plane.  \textit{J. Sing.} {\bf 10} (2014), 92--107.

\bibitem{FuTa-Fs} T. Fukunaga, M. Takahashi, Framed surfaces in the Euclidean plane. \textit{Bull. Braz. Math. Soc. (N.S.)} {\bf 50} (2019),  37--65.

\bibitem{Gray} A. Gray, E. Abbena, S. Salamon, \textit{Modern differential geometry of curves and surfaces with Mathematica. Third edition.}  Studies in Advanced Mathematics. Chapman and Hall/CRC, Boca Raton, FL, 2006.

\bibitem{Ha-Ho-Mo} Y. Hattori, A. Honda, T. Morimoto, Bour's Theorem for helicoidal surfaces with singularities. \textit{Diff. Geom. App.} {\bf 99} (2025),  102248.

\bibitem{martins-saji} 
	L. F. Martins, K. Saji,
Geometric invariants of cuspidal edges. 
	{\it Canad. J. Math}. {\bf 68} (2016), 445--462.

\bibitem{sigmaedge} 	
L. F. Martins, K. Saji, S. P. Santos, K. Teramoto,
{\it  Boundness of geometric invariants near a singularity which is a suspension of a singular curve},
Rev. Unión Mat. Argent {\bf 67} (2024), 475–502.

\bibitem{msst} 	
L. F. Martins, K. Saji, S. P. Santos, K. Teramoto,
Singular surfaces of revolution with prescribed unbounded mean curvature,
{\it An. Acad. Brasil. Ci\^enc.} {\bf 91} (2019), no. 3, e20170865, 10 pp.

\bibitem{Na-Sa-Shi-Ta} N. Nakatsuyama, K. Saji, R. Shimada, M. Takahashi, Singularities of helicoidal surfaces of frontals in the Euclidean space,  \textit{ArXiv preprint}, 2024. DOI: 10.48550/arXiv.2410.20085.

\bibitem{suy}	K. Saji, M. Umehara, K. Yamada,
 The geometry of fronts,
{\itshape Ann. of Math.} {\bf 169} (2009), 491--529.


\bibitem{TT}  M. Takahashi, K. Teramoto, Surfaces of revolution of frontals in the Euclidean space,  \textit{Bull. Braz. Math. Soc.},  New Series {\bf 51}  (2020), 887-914.

\bibitem{TY} M. Takahashi, H. Yu, On generalised framed surfaces in the Euclidean space. {\it AIMS Math.}  {\bf 9} (2024), 17716-17742.

\bibitem{Tera-parallel} K. Teramoto,Parallel and dual surfaces of cuspidal edges. \textit{ Differ. Geom. Appl.} {\bf 44} (2016),  52-62.

\bibitem{Tera-principal}  K. Teramoto, Principal curvatures and parallel surfaces of wave fronts. {Adv. Geom. }  {\bf 19} (2016),  541–554.

\bibitem{Tera-focal} K. Teramoto,  Focal surfaces of wave fronts in the euclidean 3-space.
 \textit{ Glasgow Math. J.}  {\bf 61} (2018), 425-440.





\end{thebibliography}
\end{document}